\documentclass[journal,letterpaper,onecolumn,twoside,nofonttune]{IEEEtran}
\IEEEoverridecommandlockouts
\usepackage[utf8]{inputenc} 
\usepackage[T1]{fontenc}
\usepackage{url}
\usepackage{ifthen}
\usepackage{cite}

\usepackage{amsmath, amsthm, amsfonts, amssymb, mathtools, mathrsfs}
\usepackage{color}
\usepackage{verbatim}
\usepackage[normalem]{ulem}
\usepackage{enumitem}
\usepackage{hyperref}

\usepackage[ruled,vlined]{algorithm2e} % for algorithm environment
\usepackage{algpseudocode}             % for pseudocode environment

\usepackage{tikz}
\usetikzlibrary{matrix,arrows,cd}

\usepackage{caption}  % do NOT load caption or subcaption with floatrow
\usepackage[compact]{titlesec}
\titlespacing{\section}{0pt}{*0.9}{*0.5}
\titlespacing{\subsection}{0pt}{*0.8}{*0.4}

\definecolor{purple}{rgb}{0.5, 0.0, 0.5}
\definecolor{dark_green}{rgb}{0.0, 0.5, 0.0}
\definecolor{mygray}{gray}{0.6}

\DeclareMathAlphabet{\mathpzc}{OT1}{pzc}{m}{it}

\newcommand{\ow}{\mathcal{O}}

\newcommand{\I}{\mathcal{I}}
\newcommand{\Uu}{\mathcal{U}}

\newcommand{\Uut}{\tilde{\Uu}}

\newcommand{\J}{\mathcal{J}}

\newcommand{\M}{\mathcal{M}}
\newcommand{\Mb}{\mathbf{M}}

\newcommand{\ellh}{\hat{\ell}}

\newcommand{\Rb}{\mathbf{R}}
\newcommand{\Rt}{{\tilde{R}}}
\newcommand{\Otil}{\tilde{O}}

\newcommand{\Ub}{\mathbf{U}}

\newcommand{\Vb}{\mathbf{V}}
\newcommand{\xb}{\mathbf{x}}

\newcommand{\xbh}{{\hat{\mathbf{x}}}}
\newcommand{\image}{\sfsty{im}}
\newcommand{\yb}{\mathbf{y}}

\newcommand{\ellb}{\bar{\ell}}

\newcommand{\ellt}{\tilde{\ell}}

\newcommand{\Gb}{\mathbf{G}}

\newcommand{\zb}{\mathbf{z}}

\newcommand{\Pib}{\mathbf{\Pi}}

\newcommand{\pit}{\tilde{\pi}}
\newcommand{\Omb}{\mathbf{\Omega}}
\newcommand{\Ombh}{\hat{\Omb}}

\newcommand{\Phib}{\mathbf{\Phi}}

\newcommand{\Scal}{\mathcal{S}}

\newcommand{\Z}{\mathbb{Z}}
\newcommand{\R}{\mathbb{R}}

\newcommand{\E}{\mathbb{E}}
\newcommand{\N}{\mathbb{N}}

\newcommand{\gb}{\mathbf{g}}
\newcommand{\gbt}{\tilde{\gb}}

\newcommand{\ga}{\acute{g}}

\newcommand{\gh}{\hat{g}}
\newcommand{\Pb}{\mathbf{P}}
\newcommand{\Pbh}{\hat{\Pb}}

\newcommand{\Sb}{\mathbf{S}}

\newcommand{\Sbwt}{\widetilde{\mathbf{S}}}
\newcommand{\Sbg}{\grave{\mathbf{S}}}
\newcommand{\Sbh}{\hat{\mathbf{S}}}
\newcommand{\Sbwh}{\widehat{\mathbf{S}}}
\newcommand{\Ab}{\mathbf{A}}

\newcommand{\Ecal}{\mathcal{E}}

\newcommand{\Abh}{\hat{\Ab}}
\newcommand{\Abwh}{\widehat{\Ab}}
\newcommand{\Abwt}{\widetilde{\Ab}}

\newcommand{\Aba}{\acute{\Ab}}

\newcommand{\Abg}{{\grave{\mathbf{A}}}}

\newcommand{\bb}{\mathbf{b}}

\newcommand{\bbh}{\hat{\bb}}

\newcommand{\bba}{\acute{\bb}}

\newcommand{\betah}{\hat{\beta}}

\newcommand{\Db}{\mathbf{D}}

\newcommand{\eb}{\mathbf{e}}
\newcommand{\Ib}{\mathbf{I}}
\newcommand{\Hb}{\mathbf{H}}

\newcommand{\Ubh}{\hat{\mathbf{U}}}

\newcommand{\Sigb}{\mathbf{\Sigma}}

\newcommand{\Thb}{\mathbf{\Theta}}

\newcommand{\epst}{{\tilde{\epsilon}}}

\newcommand{\deltatil}{\tilde{\delta}}

\newcommand{\Hbbar}{\bar{\mathbf{H}}}

\newcommand{\rpm}{\raisebox{.2ex}{$\scriptstyle\pm$}}

\newcommand{\sfsty}[1]{\ensuremath{\mathsf{#1}}}  % sf-style

\newcommand{\Ll}{\mathscr{L}}

\DeclarePairedDelimiter\ceil{\lceil}{\rceil}
\DeclarePairedDelimiter\floor{\lfloor}{\rfloor}

\DeclareMathOperator{\tr}{tr}

\DeclareMathOperator{\nnz}{nnz}
\DeclareMathOperator{\Var}{Var}
\DeclareMathOperator{\Bern}{Bern}

\definecolor{Green1}{rgb}{0.0, 0.5, 0.0}
\hypersetup{colorlinks,linkcolor={Green1},citecolor={blue},urlcolor={gray}}

\newtheorem{Thm}{Theorem}
\newtheorem{Cor}{Corollary}
\newtheorem{Prop}{Proposition}
\newtheorem{Lemma}{Lemma}
\newtheorem{Def}{Definition}
\newtheorem{Rmk}{Remark}

\begin{document}
\title{Distributed Hybrid Sketching for $\ell_2$-Embeddings}

\author{
  \IEEEauthorblockN{\textbf{Neophytos Charalambides}, \textbf{Arya Mazumdar}}\\
  \IEEEauthorblockA{Department of CSE and Halicio\u glu Data Science Institute, University of California, San Diego\\
  \thanks{Part of this work (namely, Sec.~\ref{distr_local_sk_sec}) was presented at the IEEE International Symposium on Information Theory, 2024 \cite{CM24}. Our main contribution here goes well-beyond \cite{CM24}, both in terms of new methods as well as generalizations.}
  Email: \texttt{ncharalambides@ucsd.edu}, \texttt{arya@ucsd.edu}}
\vspace{-5mm}
}

\maketitle

\begin{abstract}
Linear algebraic operations are ubiquitous in engineering applications, and arise often in a variety of fields including statistical signal processing and machine learning. With contemporary large datasets, to perform linear algebraic methods and regression tasks, it is necessary to resort to both distributed computations as well as data compression. In this paper, we study \textit{distributed} $\ell_2$-subspace embeddings, a common technique used to efficiently perform linear regression. In our setting, data is distributed across multiple computing nodes and a goal is to minimize communication between the nodes and the coordinator in the distributed centralized network, while maintaining the geometry of the dataset. Furthermore, there is also the concern of keeping the data private and secure from potential adversaries.  In this work, we address these issues through randomized sketching, where the key idea is to apply distinct sketching matrices on the local datasets. A novelty of this work is that we also consider \textit{hybrid sketching}, \textit{i.e.} a second sketch is applied on the aggregated locally sketched datasets, for enhanced embedding results. One of the main takeaways of this work is that by hybrid sketching, we can interpolate between the trade-offs that arise in off-the-shelf sketching matrices. That is, we can obtain gains in terms of embedding dimension or multiplication time. Our embedding arguments are also justified numerically.
\end{abstract}

% - - - - - - - - - - - - - - -
\section{Introduction}
\label{intro}

Randomization in numerical linear algebra and data science has been a key tool for dimensionality reduction of large matrices and datasets over the past 25 years \cite{Mah11,Woo14,DM16,DM17}. This is an interdisciplinary field which is abbreviated to ``RandNLA''. One of the main tools in the field is ``sketching'', a term that refers to random linear dimension reduction maps. This method provides an effective and computationally efficient randomized approach to addressing problems like matrix factorization \cite{MRT11}, eigenvalue computation \cite{SW23}, $k$-means \cite{BZMD14,CEMMP15}, or solving linear systems \cite{Sar06,DMMS11}, which are prohibitive in high dimensions and expensive with deterministic methods. The core idea behind sketching is to leverage randomization to create structured and well-conditioned matrices that preserve important properties of the original matrices after compressing them. By applying known algorithms to the sketched matrices, one obtains an approximate solution that is statistically close to the true solution of the problem. A sufficient condition which guarantees that through sketching we obtain a high quality solution to overdetermined systems is the $\ell_2$-\textit{subspace embedding} ($\ell_2$-s.e.) property, which says that with high probability the geometry of the system's basis is preserved after performing the sketching procedure. {A widely used class of sketching matrices satisfying this property are ``oblivious subspace embeddings'' (OSEs) which are dense randomized linear maps that project vectors to a lower-dimensional space.}
%An oblivious subspace embedding (OSE) is a randomized linear map that projects vectors into a lower-dimensional space while approximately preserving their norms within a target subspace, without prior knowledge of the subspace.

More recently, with the increase in daily data generation which is prevalent in many machine learning and statistical inference models, resorting to distributed systems for storage and computations is a necessity \cite{LA20,CLR23}. In many applications the data is generated locally and should not be shared with anyone within the distributed network before encrypting it, as privacy and security concerns may arise, \textit{e.g.} hospitals in a health care network each of which are comprised of local servers and are collectively operated by a health organization. A prime example with such a framework which has been extensively studied is that of federated learning (FL) \cite{MMRHA17,KMRR16,SRRA22,SGZC24}. To this end, \textit{distributed sketching} has been studied in several settings \cite{KSD17,CMPH22,BP23,BBGL23,CPH23b}, in which a central aggregating server administers a centralized distributed computing network comprised of $k$ servers with whom the coordinator communicates part of the data, and each of them locally perform a sketch or computation that is sent back. The coordinator then decides on a final sketch of the entire dataset or a solution to an optimization problem. We illustrate such a setting in Fig.~\ref{fig_schematic}, which we study.

%A recent proposal in distributed sketching, is that of \textit{hybrid sketching} \cite{BP23}.
Recently, \textit{hybrid sketching} has been proposed in distributed sketching \cite{BP23}. This refers to a compression method where a sequential application of two different sketches takes place, which is useful in distributed computing. In particular, it might be computationally feasible for server nodes to sample as much data as possible, though the central administrator wishes to further compress the global sketch in order to further accelerate the final computation. In such a scenario the server nodes should send a sketch that is not too small, in order to maintain as much geometric information of their data as possible, as the second sketch will further distort the sketched local data.

The main benefits of hybrid sketching are that we can obtain high quality sketches when: ($\mathrm{I}$) the data is either inherently distributed and too large to gather into one place or store at a single server before sketching takes place; ($\mathrm{II}$) the locally generated data is sensitive, and local sketching can provide a level of privacy. Through hybrid sketching we can maximize the information of the local data retained when bandwidth and other network constraints are imposed, so that the resulting sketch can have better approximation guarantees compared to distributed local sketching. Throughout the article, other benefits of hybrid sketching will also be discussed.

In this paper, we focus on the task of \textit{distributed} $\ell_2$-\textit{subspace embedding} through sketching, as well as its composition with \textit{hybrid sketching} for enhanced embedding results of distributed data matrices. The main contributions of our work are that we show how we can obtain a global $\ell_2$-s.e. through decentralized sketching, and  provide the first theoretical analysis of hybrid sketching. We concretely show that in certain cases hybrid sketching is in fact the appropriate sketching technique to consider. Moreover, our proposed framework ensures that no information about the data can be revealed when potential eavesdroppers are present, nor to the central coordinator. That is, we ensure security of the local data, which is desirable in FL. We also discuss the implication our sketching technique has in distributed gradient descent (GD), arguably the most common technique used in machine learning.

\subsection{Contributions}

The contributions of this paper are the following. $(\mathrm{i})$ First, we present a general framework for a global $\ell_2$-s.e. through local sketching (Thm.~\ref{subsp_emb_thm_global}), which is more general than related works \cite{TZCC22}. $(\mathrm{ii})$ We then conduct a rigorous analysis of \textit{hybrid sketching} (Thm.~\ref{hybrid_thm}), which was only considered empirically in \cite{BP23}. $(\mathrm{iii})$ To this extent, we show that by sequentially applying two distinct $\ell_2$-s.e. sketches we obtain a composite sketching matrix, which applies beyond our distributive hybrid sketching framework. $(\mathrm{iv})$ Finally, what is arguably the most interesting contribution of this work in terms of theoretical results, is that by a judicious choice of locally sketching through \textit{Subsampled Randomized Hadamard Transforms} (SRHTs) and by then applying a \textit{Rademacher sketching matrix} on the aggregated sketch globally through, it is possible to benefit in terms of multiplication time and/or reduced dimension in comparison to other prevalent sketching matrices (Sec.~\ref{subsec_SRHT_Rrad}). This is a consequence of the fact that we use fast OSEs locally, and an OSE with better embedding properties globally. This approach is also numerically superior in terms of the $\ell_2$-s.e. error.
%Furthermore, we also observed that this approach is numerically superior in terms of the $\ell_2$-s.e. error.

The motivation of our proposed hybrid scheme, was to concurrently satisfy the following desiderata, to further accelerate computations that rely on large distributed datasets: $(\mathrm{a})$ \textit{security} in distributed data aggregation, $(\mathrm{b})$ \textit{local compression} of data aggregation while maintaining the geometric properties of the global dataset, $(\mathrm{c})$ perform the compression \textit{efficiently} and \textit{distributively}, $(\mathrm{d})$ reduce the \textit{communication load} required to deliver the local data, and $(\mathrm{e})$ reduce the \textit{space requirement} required by the central server. Although for our theory we focus on least squares, a fundamental primitive tool used for solving a variety of optimization and estimation problems, our approach is applicable in other downstream applications which we numerically validate.

%\noindent\textbf{Contributions Beyond \cite{CM24}:} {Sec.~\ref{distr_local_sk_sec} explained in contribution $(\mathrm{i})$, has been published in \cite{CM24}. All other technical material in this paper is new and has not previously been presented nor published. As was mentioned above, the most interesting part of this work is that through hybrid sketching it is possible to benefit in terms of operation count and/or reduced dimension in comparison to other prevalent techniques.}

\subsection{Related Work}
\label{related_work}

The closest works to what we present, are distinct sketching techniques from \cite{TZCC22,BBGL23,BP23}. In a recent monograph \cite{murray2023}, the work of \cite{BBGL23}  is the only one mentioned in which local sketching is performed. In \cite{BBGL23}, which coined the term ``\textit{block-SRHT}'', the authors assume that a partition of the global dataset is sent to each of the nodes, who then apply a SRHT to their allocated partition along with an additional global permutation and a local signature matrix, before sending back the sketch to the coordinator who sums the sketches.

The main technique used in the block-SRHT is that by applying the global permutation and an additional signature matrix, and performing what they call the ``sum-reduce'' operation, there is resemblance with the standard SRHT \cite{AC06,AC09,DMMS11}. Their overall computational cost is a constant factor higher than sketching with a standard SRHT for the same embedding dimension. On the contrary, the overall computational cost of our approach is a constant factor lower than standard sketches, as we perform multiple smaller sketches and then \textit{aggregate} them, instead of \textit{summing} them. To this extent, our objective is also more general, as we assume the data is already distributed across a network (such as in FL), and the global dataset is never itself aggregated. Another drawback of the block-SRHT, is that its overall communication load is a factor of $k$ higher than ours, for the same embedding dimension.

The work of \cite{TZCC22} develops a technique that the authors term ``\textit{Sparse-SRHT}'' for linear regression, which resembles the algorithm we present in Sec.~\ref{distr_local_sk_sec}. Our algorithm can be viewed as a generalization of theirs, as they only consider local sketching performed through a regular SRHT, and do not consider the distributed setting depicted in Fig.~\ref{fig_schematic}. Furthermore, they assume that the data is centrally administered and do not leverage the fact that the computation, which is true also in their case, can be divided and performed in parallel by $k$ server nodes. We do however use a result of theirs to simplify our comparisons in hybrid sketching, primarily to avoid the use of an oversampling parameter, which is an artifact of the techniques we are using for our distributed algorithm. One difference between our approach and theirs, is that they also consider an additional permutation matrix that is performed on the global dataset $\Ab$, which requires either further communication between the servers or that the dataset is aggregated before local sketching takes place. As they explain, this does not affect the theoretical analysis, though we observed that experimentally such a permutation improves the flattening of the transformed leverage scores.

To the best of our knowledge, hybrid sketching has only previously been considered in the work of \cite{BP23}. The work of \cite{BP23} considers distributive iterative Hessian sketching, and the hybrid sketch they propose is simply applying uniform sampling locally, and then a Sparse Johnson-Lindenstrauss Transform (SJLT) \cite{NN13}. The main objective in their case for performing hybrid sketching is to reduce the overall complexity of the sketching procedure, and they are not concerned about the precise embedding guarantees. Furthermore, this approach is purely empirical, as uniform sampling cannot provide spectral guarantees unless very stringent assumptions are in place, which is the trade-off we pay by applying random projections. This seems to be why no theoretical justification or reasoning is given for their proposed hybrid sketch.

Other related work include distributed sketching techniques in which different sketches are performed on the global dataset to solve varying sketched versions of the global system of linear equations, which are then averaged or aggregated to determine a good approximate solution for linear regression \cite{KSD17,BP23}. The main drawback here compared to our approach is that each sketch is performed on the global dataset, and therefore it is not suitable for systems in which local nodes may wish to not share their information with other nodes.

Within the FL context, closely related works are \cite{NVR22,AN24} which also make coherence assumptions we bypass through our techniques. This is also the case for works in optimization \cite{NDV19,GMM20}. Moreover, in signal processing there have been works which consider sketching structured matrices for specific applications, \textit{e.g.} Kalman-filtering \cite{BG17} and Toeplitz matrices \cite{QP17}. In what we present we do not assume any structure on the data matrices, and our proposals can be further utilized in other applications; such as distributed covariance estimation \cite{WH12}.

\noindent\textbf{Organization:} The paper is organized as follows. In Sec.~\ref{problem_setup}, we recall the definition of an $\ell_2$-s.e., and give an outline of our proposed scheme which implements such an embedding distributively without aggregating the data a priori. 
%In Sec.~\ref{related_work}, we overview closely related works and applications. 
In Sec.~\ref{randnla_prelim} we setup the necessary notation and review background from RandNLA which we will need for the main portion of the paper. We then move to the main body, where in Sec.~\ref{distr_local_sk_sec} we formally present our algorithm for distributed local sketching and how it obtains a global $\ell_2$-s.e. The analyses of our results are presented in Sec.~\ref{analysis_subsec}. In Sec.~\ref{hybr_sketching_sec} we extend the work of Sec.~\ref{distr_local_sk_sec} to hybrid sketching. In Sec.~\ref{sec_distr_GD_sec} we discuss the implications of our proposed methods to distributed GD, as well as their security and privacy aspects. Finally, we present numerical experiments in Sec.~\ref{exper_sec_distr_sketching} and Sec.~\ref{exper_sec}, and concluding remarks in Sec.~\ref{concl_sec}. All proofs can be found in Appendix \ref{app_proofs}, and in Appendix \ref{discr_gaussian_appendix} we discuss an intuitive justification for our proposed hybrid approach.

% - - - - - - - - - - - - - - -
\section{Problem Setup and Our Distributed Scheme}
%\subsection{}
\label{problem_setup}

One of the most representative applications of RandNLA is the linear least squares problem, in which we seek to approximately solve the overdetermined linear system $\Ab\xb=\bb$:
\begin{equation}
\label{x_star_pr_lr}
  \xb^{\star} = \arg\min_{\xb\in\R^d}\Big\{L_{ls}(\Ab,\bb;\xb)\coloneqq\|\Ab\xb-\bb\|_2^2\Big\}
\end{equation}
for $\Ab\in\R^{N\times d}$ and $\bb\in\R^{N}$, where $N\gg d$. This is foundational method in data fitting and optimization, which has applications in a variety of fields. A regularizer $\lambda T(\xb)$ can also be added to $L_{ls}(\Ab,\bb;\xb)$ if desired. In our setting, we assume that $\Ab$ and $\bb$ partitioned across their rows are local datasets $\Ab_i$ with corresponding labels $\bb_i$, \textit{i.e.}
\begin{equation}
\label{part_data}
  \Ab=\Big[\Ab_1^\top \ \cdots \ \Ab_k^\top\Big]^\top \quad \text{ and } \quad \bb=\Big[\bb_1^\top \ \cdots \ \bb_k^\top\Big]^\top
\end{equation}
where $\Ab_i\in\R^{n\times d}$ and $\bb_i\in\R^{n}$ for all $i$, and $n=N/k$ for which $n>d$. We consider the reduced SVD of $\Ab=\Ub\Sigb\Vb^\top$, where $\Ub\in\R^{\N\times d}$ and $\Ab$ is full rank. To simplify our presentation we assume that $k|N$, and $\{\Ab_\iota\}_{\iota=1}^k$ are equipotent. A way to approximate \eqref{x_star_pr_lr} in a faster manner, is to instead solve the modified least squares problem
\begin{equation}
\label{x_til_pr_lr}
  \xbh = \arg\min_{\xb\in\R^d}\Big\{L_{\Sb}(\Ab,\bb;\xb)\coloneqq\|\Sb(\Ab\xb-\bb)\|_2^2\Big\}\
\end{equation}
for $\Sb\in\R^{R\times N}$ an $\ell_2$-s.e. sketching matrix, with $R<N$.

\begin{Def}[Ch.2 \cite{Woo14}]
\label{se_def}
A sketching matrix $\Sb\in\R^{R\times N}$ is a \textbf{$\ell_2$-subspace embedding} (or satisfy the \textbf{$\ell_2$-s.e. property}) of $\Ab\in\R^{N\times d}$ with a left orthonormal basis $\Ub$, and  $N\gg d$ s.t. $N>R>d$, if for any $\yb\in\image(\Ub)$ we have with high probability:
\begin{equation}
\label{subsp_emb_def}
  %\|\Sb\yb\|_2\leqslant_\epsilon\|\yb\|_2  \ \iff \ 
  (1-\epsilon)\cdot\|\yb\|\leqslant\|\Sb\yb\|\leqslant(1+\epsilon)\cdot\|\yb\|
\end{equation}
for $\epsilon>0$. The $\ell_2$-s.e. property is equivalent to satisfying:
\begin{equation}
\label{subsp_emb_id}
  \|\Ib_d-(\Sb\Ub)^\top(\Sb\Ub)\|_2\leqslant \epsilon\ .
\end{equation}
\end{Def}

In turn, Definition \ref{se_def} characterizes the approximation's error of the solution $\xbh$ of \eqref{x_til_pr_lr} as
%$$ \|\Ab\xbh-\bb\|_2 \leqslant \frac{1+\epsilon}{1-\epsilon}\|\Ab\xb^{\star}-\bb\|_2 \leqslant (1+\ow(\epsilon))\|\Ab\xb^{\star}-\bb\|_2 $$
\begin{equation}
\label{epsilon_error_LS}
  \|\Ab\xbh-\bb\|_2 \leqslant (1+\ow(\epsilon))\|\Ab\xb^{\star}-\bb\|_2
\end{equation}
and $\|\Ab(\xb^{\star}-\xbh)\|_2\leqslant\epsilon\|(\Ib_N-\Ub\Ub^\top)\bb\|_2$ \cite{ERNM22}. Furthermore, desired properties of sketching matrices are that they are zero-mean and normalized, \textit{i.e.} $\E[\Sb]=\bold{0}_{R\times N}$ and $\E[\Sb^\top\Sb]=\Ib_N$ \cite{RE21} to guarantee unbiased estimates, which are met by normalized random matrices with i.i.d. standard Gaussian entries. Intuitively, random matrices $\Sb$ (\textit{e.g.} Gaussian, Rademacher, SRHT) acts like scramblers, and spread out information across the rows of $\Sb\Ab$ uniformly. When $\Sb$ is applied to $\Ab$, the information in any low-dimensional subspace of $\R^n$ does not collapse nor is significantly distorted.

In terms of $\ell_2$-s.e., we show that it is possible to aggregate sketches of local datasets from a distributed network, to obtain a sketch of the collective dataset as a whole (referred to as the ``\textit{global dataset}''). We work with OSEs, a category of $\ell_2$-s.e. which suit our objective. Specifically, by performing a random projection on the local datasets $\{\Ab_i\}_{i=1}^k$, we show that in the resulting global dataset $\Ab$ the transformation of each data point is of approximately equal importance, which is quantified by leverage scores. Thereby, this implies that local uniform sampling suffices. Embeddings in such data-distributed settings are useful in a variety of data pre-processing/training and unsupervised learning tasks such as clustering and PCA~\cite{guha2000clustering,gandikota2020reliable,ZV15,TSG15,TG18}.

To obtain the sketch of the global dataset, the distributed nodes send their local sketches $\{\Sb_i\Ab_i\}_{i=1}^k$ to a coordinator who aggregates them, to obtain a global sketch $\Abwh\in\R^{R\times d}$ where $R\ll N$. Depending on the choice of the random projection performed by the nodes, there is a security guarantee on their local information which prohibits potential eavesdroppers and the coordinator from recovering  the data points, which is of increasing importance in distributed machine learning. Ultimately, we obtain a summary of the global dataset in a  decentralized manner, without explicitly revealing or aggregating the local data. This approach is depicted pictorially in Fig.~\ref{fig_schematic}.

\begin{figure}[h]
  \centering
    \includegraphics[scale=.18]{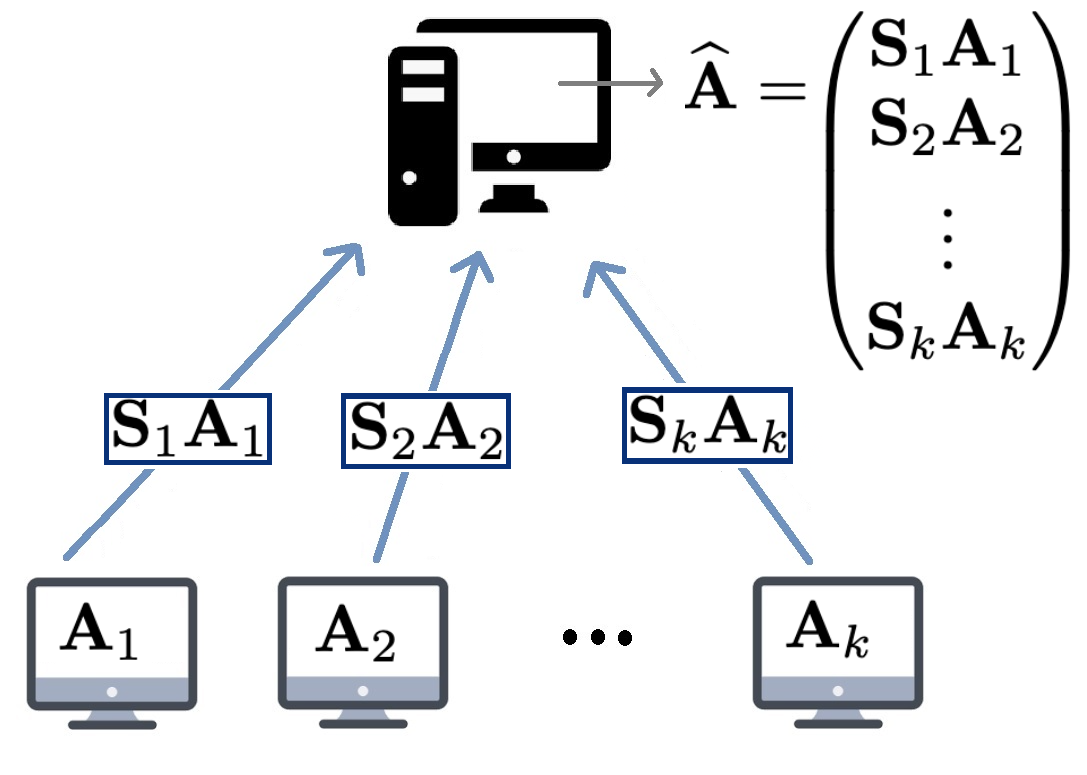}
    \caption{Schematic of distributed aggregated sketching.}
  \label{fig_schematic}
\end{figure}

In our work, a second sketching matrix $\Phib$ is applied to the global sketch $\Abwh$, to obtain further compression through hybrid sketching. Depending on the combination of sketching matrices and projections considered for the pairs $\Sb_{\{k\}}$ and $\Phib$, hybrid sketching may lead to better performance results than when applying a single sketching matrix of the same final target dimension. Specifically, the ideal combination is to first apply a fast sketching matrix (which in our case is done distributively through local sketching), and then a second one with better properties in terms of dimensionality reduction. This technique leads to further benefits on distributed sketching, and is an improvement on the work of \cite{BP23} which only considered uniform sampling for $\Sb_{\{k\}}$ and did not provide any theoretical justification to their approach. Surprisingly, in numerical experiments there was a clear distinction in terms of the empirical embedding guarantee of hybrid sketching versus standard sketching techniques.

In Table \ref{table_OSE} we summarize the most relevant results regarding OSE sketches, which we will use in Sec.~\ref{hybr_sketching_sec} for comparison to hybrid sketching, and to show when there is an improvement through our proposal. We also include sketching through leverage score sampling for completeness, which is the main tool we use for proving Theorem \ref{subsp_emb_thm_global}, though we do not consider sparse sketches. For further details on these sketching techniques, refer to the monographs \cite{Woo14,HMJ11,Mah11,Mah16,Wan15,DM17,murray2023,Nelson20,Ver18,MT20}.

\begin{center}
\begin{table*}[t]
\centering
\begin{tabular}{ |p{4.2cm}||p{4.2cm}|p{3.5cm}| }
\hline
\multicolumn{3}{|c|}{\textbf{Summary of Main Sketching Techniques}} \\
\hline
\hline
\textbf{Sketch Technique} & \textbf{Embedding Dimension $R$} & \textbf{Multiplication Time} \\
\hline
\hline
Gaussian Sketch \cite{KVZ14,Ver18} & $\ow\big((d+\log(1/\delta))/\epsilon^2\big)$ & $\ow(RNd)$ \\ \hline
ROS \cite{PW16,CMPH23} & $\Theta\big(d\log(d/\delta)/\epsilon^2\big)$ & $\ow(RNd)$ \\ \hline
SRHT \cite{AC09,DMMS11} & $\Theta\big((d\log(Nd/\delta)\log(d/\delta)/\epsilon^2\big)$ & $\ow\big(Nd\log(N)\big)$ \\ \hline
Rademacher Sketch \cite{CW09,TYUC17} & $\ow\big((d+\log(1/\delta))/\epsilon^2\big)$ & $\ow(RNd)$ \\ \hline
Uniform Sampling \cite{MT20} & $\ow\big(N\gamma(\Ab)\log(1/\delta)/\epsilon^2\big)$ & $\ow(Rd)$ \\ \hline
Leverage Scores \cite{ERNM22,CPH23b} & $\Theta\big(d\log(d/\delta)/(\beta\epsilon^2)\big)$ & $\ow(Rd)$ \\ \hline
\textit{Hybrid} (Radem. $\&$ SRHTs) & $\ow\big((d+\log(1/\delta))/\epsilon^2\big)$ & $\ow\big(Nd(\log(n)+N\mu^2\rho)\big)$ \\ \hline %$\ow\big(Nd(\log(n)+N\rho\mu^2)\big)$
\end{tabular}
\caption{The $1^{st}$ column indicates the sketching technique. The $2^{nd}$ column shows the reduced embedding dimension of the resulting sketch, \textit{i.e.} $R$ for $\Sb\in\R^{R\times N}$ required for \eqref{subsp_emb_id} to be satisfied. The $3^{rd}$ column indicates the multiplication time it takes to apply $\Sb$ on $\Ab$, with naive matrix-matrix multiplication. ROS in the $3^{rd}$ row stands for ``Randomized Orthonormal Systems''. Parameter $\beta\in(0,1]$ in the $6^{th}$ row, is an misestimation factor that accounts for sampling according to an approximate distribution. The compression factors of the first and second sketch respectively in hybrid sketching, are respectively $\mu,\rho\in(0,1)$.}
\label{table_OSE}
\end{table*}
\end{center}

% - - - - - - - - - - - - - - -
\section{Sketching Overview}
\label{randnla_prelim}

We denote $\N_m\coloneqq \{1,2,\ldots,m\}$, and $X_{\{m\}} = \{X_i\}_{i=1}^m$ where $X$ can be replaced by any variable. We consider random square projection matrices, which are represented by $\Pb$. Sampling matrices are denoted by $\Omb$, where each row has a single nonzero entry. The uniform sampling distribution $\{1/N,\ldots,1/N\}$ is denoted by $\Uu_N$, and by $\Uut_N$ we represent $\Uu_N$'s approximation through our approach. The index set of $\Ab$'s rows is denoted by $\I$, \textit{i.e.} $\I=\N_N$. The index set of $\Ab_i$ is denoted by $\I_i$ for each $i\in\N_k$, \textit{i.e.} $\I_i=\{(i-1)n+1,\ldots,in\}$ and $\I=\bigsqcup_{\iota=1}^k\I_\iota$. Similarly, by $\Scal$ we denote the index multiset of $\Ab$'s sampled rows, and $\Scal_i$ the index multiset of $\Ab_i$'s sampled rows for each $i$. By $\eb_j$ we denote the $j^{th}$ standard basis vector. The restriction of a matrix $\Mb$ to the entries of $\I_\iota$ is represented by $\Mb\big|_{\I_\iota}$, \textit{e.g.} $\Ib_n\big|_{(\{1,2\})}=(\eb_1\eb_1^\top+\eb_2\eb_2^\top)\in\R^{n\times n}$ . The $j^{th}$ row of matrix $\Mb$ is represented by $\Mb_{(j)}$. We abbreviate $(1-\epsilon)\cdot b\leqslant a \leqslant(1+\epsilon)\cdot b$ to $a\leqslant_\epsilon b$, and the complementary event of $|a-b|>\epsilon\cdot b$ to $a>_\epsilon b$. We use ``$\wedge$'' and ``$\sim$'' (over-scripts) to denote corresponding quantities of $\Ab$ in the global $\Abwh$ and hybrid $\Abwt$ sketches respectively.% The nuclear norm of $\Mb$ is defined as $\|\Mb\|_\star=\sum_{i=1}^{\rk(\Mb)}\sigma_i(\Mb)$, where $\rk(\Mb)$ and $\sigma_i(\Mb)$ are respectively the rank and singular values of $\Mb$, and the nuclear rank of $\Mb$ is defined as $\nr(\Mb)\coloneqq\|\Mb\|_\star/\|\Mb\|_2$.

% - - - - - - - - - - - - - - -
\subsection{Sketching through Leverage Scores}
\label{sk_thr_lvg_scores_sec}

Many sampling algorithms select data points according to the data's \textit{leverage scores} \cite{MMY15,DMMW12}. The leverage scores of $\Ab$ measure the extent to which the vectors of its orthonormal basis $\Ub$ are correlated with the standard basis, and define the key structural non-uniformity that must be dealt with when developing fast randomized matrix algorithms, as they characterize the importance of the data points. Leverage scores are defined as $\ell_j\coloneqq\|\Ub_{(j)}\|_2^2$ and are agnostic to any particular basis, as they are equal to the diagonal entries of the projection matrix $P_\Ab=\Ab\Ab^{\dagger}=\Ub\Ub^\top$. The \textit{normalized leverage scores} of $\Ab$ for each $j\in\N_N$ are
\begin{equation*}
\label{norm_lvg_sc}
  \pi_j \coloneqq \left\|\Ub_{(j)}\right\|_2^2\big/\|\Ub\|_F^2 = \left\|\Ub_{(j)}\right\|_2^2\big/d\ ,
\end{equation*}
and $\pi_{\{N\}}$ form a sampling probability distribution, as $\sum_{\iota=1}^N\pi_\iota=1$ and $\pi_j\geqslant0$ for all $j$. This induced distribution has been proven useful in linear regression \cite{DMMW12,Woo14,Mah16,ERNM22}, as well as a plethora of other applications \cite{SS11,BZMD14,DM16,OMG22}.

The \textit{coherence} of $\Ab$ is defined as $\gamma(\Ab)\coloneqq\max_{\iota\in\N_N}\big\{\ell_\iota\big\}$. Similar to \eqref{part_data}, we partition% Considering the partitioning $\Ab$ according to \eqref{part_data}, we partition
\begin{equation}
\label{part_U}
  \Ub=\Big[\Ub_1^\top \ \cdots \ \Ub_k^\top\Big]^\top
\end{equation}
and define the ``\textit{local coherence}'' of each data block as $\gamma_i\coloneqq\max_{j\in\I_i}\big\{\ell_j\big\}$. We denote the sum of each block's corresponding leverage scores by $\Ll_i\coloneqq\sum_{j\in\I_i}\ell_j$. It is worth noting that in our setting, the closer $\Ll_i$ and $\gamma_i$ are to $d/k$ and $d/N$ respectively, the more homogeneous the local data blocks are when they considered as a global dataset. If $\gamma_i=d/N$, $\Ub_i$ is aligned with the standard basis.

Next, we recall the leverage score sampling $\ell_2$-s.e. sketch, which is used to obtain $\Abg\coloneqq\Sbg\Ab$. Given $\Ab$, we sample $R>d$ rows \textit{with replacement} (w.r.) from $\Ab$ according to $\pi_{\{N\}}$. If at trial $j$ the row $i_j$ was sampled we rescale it by $1/\sqrt{R\pi_{i_j}}$, and set $\Abg_{(j)}=\Ab_{(i_j)}\big/\sqrt{R\pi_{i_j}}$. It is clear that here $\Sbg\in\R^{R\times N}$ is simply a sampling and rescaling matrix, \textit{i.e.} $\Sbg_{j}=\eb_{i_j}^\top\big/\sqrt{R\pi_{i_j}}$.

In many cases, estimating the leverage scores is preferred, as computing them exactly requires $\ow(Nd^2)$ time which is excessive. We can instead use accurate approximate scores $\ellt_{\{N\}}$ which can be computed in $\ow(Nd \log N)$ time \cite{DMMW12}.
The estimates are ``close'' in the following sense: $\ellt_i\geqslant\beta\ell_i$ for all $i$, where $\beta\in(0,1]$ is a misestimation factor. The only difference in sampling according to $\ellt_{\{N\}}$, is that we need to oversample by a factor of $1/\beta$ to get the same theoretical guarantee. The $\ell_2$-s.e. result of $\Sbg$ is presented next \cite{Woo14,ERNM22,CPH23b}.%$\Abg_{(j)}=\frac{1}{\sqrt{R\pi_{i_j}}}\Ab_{(i_j)}$  $\Sbwh_{j}=\frac{1}{\sqrt{R\pi_{i_j}}}\eb_{i_j}^\top$

\begin{Thm}
\label{lvg_score_se_thm}
The leverage score sketching matrix $\Sbg$ is a $\ell_2$-\textit{s.e} of $\Ab$. Specifically, for $\delta>0$ and $R=\Theta\left(d\log{(2d/\delta)}/(\beta\epsilon^2)\right)$, the identity of $\eqref{subsp_emb_id}$ is satisfied with probability at least $1-\delta$.
\end{Thm}

% - - - - - - - - - - - - - - -
\subsection{Oblivious Subspace Embeddings}
\label{ose_sec}

A drawback of directly applying leverage score sampling locally in hope of obtaining a global sketch by aggregating the local sketches, is that the local datasets may be highly heterogeneous, and the sampling performed locally may not be representative of the global leverage score sampling distribution. To alleviate this issue, we resort to OSEs, which exploit random projections and/or uniform sampling.

\begin{Def}
\label{OSE_def}
%https://arxiv.org/pdf/2411.08773
A Random matrix $\Pib\in\R^{R\times N}$ is an $(\epsilon,\delta,d)$-\textbf{oblivious subspace embedding} (OSE) if for any $d$-dimensional subspace $T=\image(\Ub)\subseteq\R^N$, it satisfies \eqref{subsp_emb_id} for $\Sb\gets\Pib$.
\end{Def}

Two prime examples of OSEs, are the Gaussian sketch and the SRHT. It is also worth noting that utilizing random Gaussian and Rademacher random matrices in data compression has close ties to the Johnson-Lindenstrauss lemma \cite{JL84,Ach03,DG03}, which predates the study of RandNLA. We note that Rademacher sketching matrices are also referred to as ``Sub-Gaussian sketches'' in the literature, since the tail of Rademacher random variables follow a sub-Gaussian decay, \textit{i.e.} its tail probabilities decay at least as fast as those of a normal distribution.

The Gaussian sketch is defined through a random projection $\Gb\in\R^{R\times N}$ where $\Gb_{ij}\sim\mathcal{N}(0,1)$, which is then rescaled to get $\Sb=\frac{1}{\sqrt{R}}\Gb$. To unify our techniques, we note that directly applying $\frac{1}{\sqrt{R}}\Gb\in\R^{R\times N}$ is equivalent to uniformly sampling (without replacement) $R$ rows from %$\frac{1}{\sqrt{R}}\Gb\in\R^{N\times N}$
a $N \times N$ Gaussian matrix. This is also true for the Rademacher sketch, where $\Thb_{ij}\sim\sfsty{Unif}(-1,+1)$ and $\Sb=\frac{1}{\sqrt{R}}\Thb$. For further speedups with these unstructured projections, one could therefore directly apply $R\times N$ rescaled projections and not consider uniform sampling. A benefit of considering the uniform sampling matrix $\Omb$ being generated separately to the random projection, is that other sampling matrices may be utilized instead \cite{NN13,DLPM21}.

The SRHT is comprised of three matrices: $\Omb\in\R^{R\times N}$ a uniform sampling w.r. and rescaling matrix of $R$ rows, $\Hbbar_N$ the normalized Hadamard matrix of order $N$:
\begin{equation*}
\label{Had_matrix}
  \Hb_N = \begin{pmatrix} 1 & 1 \\ 1 & -1 \end{pmatrix}^{\otimes \log_2(N)} \qquad \Hbbar_N = \frac{1}{\sqrt{N}}\cdot\Hb_N
\end{equation*}
and $\Db\in\{0,\rpm1\}^{N\times N}$ with i.i.d. diagonal Rademacher random entries; \textit{i.e.} it is a signature matrix. If $N$ is not a power of 2, we can pad $\Ab$ with zeros to meet this requirement. The SRHT sketching matrix is then $\Sb=\sqrt{\frac{N}{R}}\cdot\Omb\Hbbar_N\Db$, where $\Hbbar_N\Db$ is a unitary matrix that rotates $\Ub$. Note that the scaling here is also consistent with the aforementioned sketching matrices, as $\Sb=\frac{1}{\sqrt{R}}\cdot\Omb\Hb_N\Db$. The main intuition of the projection is that it expresses the original signal or feature-row in the Walsh-Hadamard basis. Furthermore, $\Hbbar_N$ can be applied in $\ow(Nd\log N)$ time, by using Fourier based methods.

In the new left orthonormal basis of $\Ab$ after the aforementioned projections are applied, the resulting leverage scores are close to uniform. Hence, uniform sampling is applied through $\Omb$ to reduce the effective dimension $N$, whilst the information of $\Ab$ is maintained. An appropriate rescaling according to the number of sampling trials also takes place, in order to reduce the variance of the resulting estimator.

The idea behind our approach is that the local projections will ``flatten'' the leverage scores in $\Ab$ of their local blocks, \textit{i.e.} $\pi_{j}\approx \Ll_i/(nd)$ for each $j\in\I_i$ and every $i\in\N_k$. By then locally performing uniform row sampling on $\Pb_i\Ab_i$, we get a close to uniform sampling across all the projected blocks.%$\pi_{j}\approx\pi_l$ for each $j,l\in\I_i$ for every $i\in\N_k$,

From Table \ref{table_OSE}, it is clear that there is a trade-off between the choice of OSEs. On the one hand, this can be applied efficiently but have a logarithmic dependence on $N$ when it comes to $R$; \textit{e.g.} SRHT. On the other hand, you have a slower multiplication which does not utilize Fourier methods, but has no dependence on $N$ when it comes to the number of rows needed to be sampled. Through our hybrid sketching proposal, \textit{we benefit from both approaches}, by first applying a faster sketch $\Pib_1$; \textit{e.g.} SRHT, and then a sketching matrix $\Pib_2$ which requires a lower $R$; \textit{e.g.} Rademacher sketch. We also have the benefit of the first sketch being applied distributively, without ever needing to aggregate all the data.

Since the resulting sketching matrix $\Sbh=\Pib_2\Pib_1$ is composed of two distinct linear transformations, we cannot hope to get an improvement individually on one of the two characteristics discussed above over all sketching techniques. We do however get a better balance on the two, when compared to the individual sketching approaches. By the aforementioned selection of $\Pib_1$ and $\Pib_2$ presented in the last row of Table \ref{table_OSE}, we therefore obtain an interpolation between the two extreme cases of OSEs. As we elaborate on in Sec.~\ref{hybr_sketching_sec}, a primary objective for applying a dense $\Pib_2$ in our distributed setting, is to obtain linear combinations of all data points in the final sketch, by a recombination of the local sketches. In this scenario, $\Pib_1$ can be interpreted as a sparse sketching matrix \eqref{global_sketch_S}.

\noindent \textbf{Comparison to Importance Sampling:} From Table \ref{table_OSE}, we note that in worst cases the $\log(d/\delta)$ factor in the leverage score sampling sketch which is non-oblivious, implies that leverage score sketches have a theoretically higher sample complexity $R$. However, this is a worst case scenario, which occurs when $\gamma(\Ab)\approx d/N$; \textit{i.e.} $\Ab$ has low coherence and the leverage scores across the rows are close to uniform. This is also the assumption that other works make to simplify their analyses, which we bypass through local projections, as was mentioned in Sec.~\ref{intro}. Moreover, the effect of leverage score sampling is beneficial when $\gamma(\Ab)$ is high, meaning that a small subset of rows is of greater importance than the rest, and require a lower sample complexity in practice to maintain the structure of $\Ab$. High leverage rows are critical for preserving the span of the matrix's principal components, allowing leverage score sketches to capture essential information with fewer samples than a Gaussian sketch, especially when $\Ab$ is incoherent. This point is captured in the case of uniform sampling for an $\ell_2$-s.e. in the $5^{th}$ row of Table \ref{table_OSE}, where the embedding dimension is proportional to the coherence of the subspace. In such cases, leverage score sketches can provide accurate results with fewer samples than the required $R$ indicated above, due to the fact that they then perform non-uniform targeted sampling, which often reduces the effective sample size required in practice. The Gaussian and Rademacher sketches which are OSEs drop this dependency on $\gamma(\Ab)$, thus uniform sampling suffices, and their analyses do not need to consider a worst case scenario.

% - - - - - - - - - - - - - - -
\section{Distributed Local Sketching}
\label{distr_local_sk_sec}

In this section, we discuss the details of our distributed sketching scheme (Algorithm \ref{alg_distr_sketch}). In the setup of Fig.~\ref{fig_schematic}, the $i^{th}$ node applies a random projection matrix $\Pb_i\in\R^{n\times n}$ which is generated locally to flatten the corresponding leverage scores. As we will see, the flattening here is with respect to (w.r.t.) $\Ub_i$; \textit{i.e.} w.h.p. $\ell_j\approx\Ll_i/(nd)$ for each $j\in\I_i$. This is the cost we pay for performing local sketching. Nonetheless, we show in Fig.~\ref{lvg_scores_tdistr_flat_fig} that the flattening degrades gracefully as $k$ increases, even for global datasets with highly non-uniform leverage scores. To partially circumvent this concern if our approach is to be performed by a single user or centrally administered by the coordinator, a random permutation can be applied on the rows of $\Ab$ before the partitioning takes place.

After locally applying $\Pb_i$, the nodes randomly sample $r=R/k$ rows from $\Pb_i\Ab_i$ which they rescale by $\sqrt{n/r}=\sqrt{N/R}$ and aggregate through $\Omb_i$, to obtain the local sketches
\begin{equation*}
\label{local_sketch_eq}
  \Sb_i\Ab_i\in\R^{r\times d} , \quad \text{for} \quad \Sb_i=\sqrt{n/r}\cdot\left(\Omb_i\cdot\Pb_i\right)\in\R^{r\times n} .
\end{equation*}
Then, each node communicates the resulting sketch to the coordinator, who aggregates them.

\begin{algorithm}[h]
\SetAlgoLined
  \KwIn{Effective local dimension $r$ \Comment{$R>d$ and $R=rk$}}
  \KwOut{Sketch $\Abwh\in\R^{R\times d}$, of the collective dataset $\Ab$}
  \For{$i=1$ to $k$}% \Comment{\underline{$i^{th}$ node}
    {
      \underline{$i^{th}$ node}:\\
      $\ 1)$ Generate a random $\Pb_i\in\R^{n\times n}$ \Comment{$\E\left[\Pb_i^\top\Pb_i\right]=\Ib_n$}\\
      $\ 2)$ Uniformly sample $r$ rows from $\Pb_i\Ab_i$, through $\Omb_i$\\% through $\Omb_i\in\{0,1\}^{r\times N}$\\
      $\ 3)$ Deliver $\sqrt{\frac{n}{r}}\left(\Omb_i\Pb_i\right)\Ab_i\eqqcolon \Sb_i\Ab_i$ to the coordinator\\
      \textit{\underline{\textbf{note}}: $1)$ and $2)$ can be performed simultaneously by generating $\Sb_i\in\R^{r\times n}$, to reduce the local computations}
    }
    \underline{Coordinator}: Aggregates $\Abwh = \Big[(\Sb_1\Ab_1)^\top \ \cdots \ (\Sb_k\Ab_k)^\top\Big]^\top$
    %Aggregate $\Abwh = \Big[\Sb_1\Ab_1^\top \ \cdots \ \Sb_k\Ab_k^\top\Big]^\top\in\R^{R\times d}$
\caption{Distributed Local Sketching}
\label{alg_distr_sketch}
\end{algorithm}

\subsection{Analysis of our approach}
\label{analysis_subsec}

For the analysis of our approach, we note that the final sketch $\Abwh$ of $\Ab$ can be summarized by the ``\textit{global sketching}'' matrix $\Sbwh\in\R^{R\times N}$, comprised of the ``\textit{local sketching}'' matrices $\Sb_{\{k\}}$ across its diagonal:
\begin{equation}
\label{global_sketch_S}
  \overbrace{\begin{pmatrix} \Sb_1 & & \\ & \ddots & \\ & & \Sb_k \end{pmatrix}}^{\Sbwh\in\R^{R\times N}} = \sqrt{\frac{N}{R}}\cdot\overbrace{\begin{pmatrix} \Omb_1 & & \\ & \ddots & \\ & & \Omb_k \end{pmatrix}}^{\Ombh\in\{0,1\}^{R\times N}} \cdot \overbrace{\begin{pmatrix} \Pb_1 & & \\ & \ddots & \\ & & \Pb_k \end{pmatrix}}^{\Pbh\in\R^{N\times N}}
\end{equation}
where the sampled index multisets $\Scal_{\{k\}}$ correspond to the sampling matrices $\Omb_{\{k\}}$, and $\bigcup_{i=1}^k\Scal_i$ to $\Ombh$.

It is noteworthy that $\Sbwh$ can also be interpreted as a sparse sketching matrix, when carried out locally by a single server. By the block diagonal structure of $\Sbwh$, the resulting global sketch recovered by the coordinator is:
\begin{equation}
\label{global_sketch}
  \Abwh \coloneqq \Sbwh\Ab = \Big[(\Sb_1\Ab_1)^\top \ \cdots \ (\Sb_k\Ab_k)^\top\Big]^\top\in\R^{R\times d} .
\end{equation}
%$$ \Abwh = \begin{pmatrix} \Sb_1\Ab_1 \\ \Sb_2\Ab_2 \\ \vdots \\ \Sb_k\Ab_k \end{pmatrix} $$

Note that $\sqrt{R/N}$ and $\Ombh$ commute. Since $\sqrt{R/N}\cdot\Pbh$ is a block diagonal matrix, the corresponding blocks of $\Ub$ are rotated/transformed by their respective projections, \textit{i.e.}:
\begin{equation*}
\label{part_Ut}
  \Ubh \coloneqq \left(\sqrt{N/R}\cdot\Pbh\right)\cdot\Ub = \Big[\Ubh_1^\top \ \cdots \ \Ubh_k^\top\Big]^\top
\end{equation*}
where $\Ubh_i=\big(\sqrt{n/r}\cdot\Pb_i\big)\Ub_i$ for each $i\in\N_k$. We denote the leverage scores of the projected matrix $\Pbh\Ab$ by $\ellh_j=\|\Ubh_{(j)}\|_2^2$.

Next, we show that the leverage scores of each $\Ubh_i$ are flattened w.r.t. $\Ll_i$ by the local projection. For our analysis, we assume that $\Pb_i$ is a random unitary matrix, drawn from an arbitrary large finite subset $\Otil_n(\R)$ of the set orthonormal matrices $O_n(\R)$ of size $n\times n$. Such sketching techniques, in which we first project with a unitary matrix and and then uniformly sample rows from the transformed matrix, are referred to as ``Randomized Orthonormal Systems'' \cite{PW16}. The SRHT is a special case and the most popular, due to to the fact the Hadamard matrix is well structured and can be applied efficiently. In practice, normalized Gaussian and Rademacher matrices are also widely used, as they satisfy $\E\left[\Gb^\top\Gb\right]=\Ib_N$. Analogous results can also be derived for these options of $\Pb_i$.%Such matrices though are not easy to generate

%It is important to note that $O_n(\R)$ is a regular submanifold of the general linear group $\GL_n(\R)$. Hence, we can define a distribution on any subset of $O_n(\R)$. For simplicity, we consider the uniform distribution. A simple method of generating a random matrix that follows the uniform distribution on the Stiefel manifold $V_m(\R^m)$ can be found in \cite[Theorem 2.2.1]{Yas12}. Alternatively, one could generate a random Gaussian matrix and then perform Gram–Schmidt to orthonormalize it.

\begin{Lemma}
\label{exp_local_scores}
Consider a fixed $i\in\N_k$. Assume that $\Pb_i$ is arbitrarily drawn from $\Otil_n(\R)$. Then, for any $j\in\I_i$, we have $\E[\ellh_j]=\Ll_i/n$.
%Consider a fixed $i\in\N_k$. Assume that $\Pb_i$ is drawn from a finite subset $\Otil_n(\R)$ of the set orthonormal matrices $O_n(\R)$ of size $n\times n$. Then; for any $j\in\I_i$, we have $\E[\ellh_j]=\Ll_i/n$.
\end{Lemma}

\begin{Prop}
\label{norm_lvg_prop}
For a fixed $i\in\N_k$ and $\xi>0$, the normalized (w.r.t. $\Ubh$) leverage scores $\{\ellb_j\}_{j\in\I_i}$ corresponding to $\Ubh_i$ satisfy
$$ \Pr\left[|\ellb_j-\Ll_i/(nd)|<\zeta\right] \geqslant 1-\xi $$ %1-2e^{-2\left(\zeta d/\Ll_i\right)^2}
for any $\zeta \geqslant \zeta' \coloneqq \frac{\Ll_i}{d}\sqrt{\log(2/\xi)/2}$.
\end{Prop}

%In Fig.~\ref{lvg_scores_tdistr_flat_fig}, we show numerically how the normalized leverage scores are flattened for $\Pb_i$ random Gaussian matrices and $\Ab$ with highly non-uniform leverage scores.
In Fig.~\ref{lvg_scores_tdistr_flat_fig}, we show numerically the flattening of the normalized leverage scores of a random $\Ab\in\R^{2000\times 40}$ following a $t$-distribution, which scores were highly non-uniform. In this experiments, $\Pb_{\{k\}}$ were random Gaussian matrices. As noted previously, random Gaussian matrices are good surrogates for unitary random matrices, and are widely used in practice as they are approximately orthogonal. Furthermore, we observe that the flattening degrades gracefully as $k$ increases, and if a random permutation on $\Ab$'s rows is applied before $\Pbh$, we have slightly better results for each $k$. Analogous simulation results were observed when the $\Pb_i$'s were randomized Hadamard transforms, random unitary, and Rademacher random matrices.

The choice for presenting the flattening of leverage scores when $\Ab\sim t$-distribution is to convey that this occurs even when the scores are highly non-uniform, which is central to our approach. Since flattening occurs w.h.p. when the data points are inhomogeneous, we expect this to also be true with an even higher probability when they are homogeneous. Further experimental verification, considering other synthetic data matrices, can be found in Appendix \ref{app_flattening}.

\begin{figure}[h]
  \centering
    \includegraphics[scale=.3]{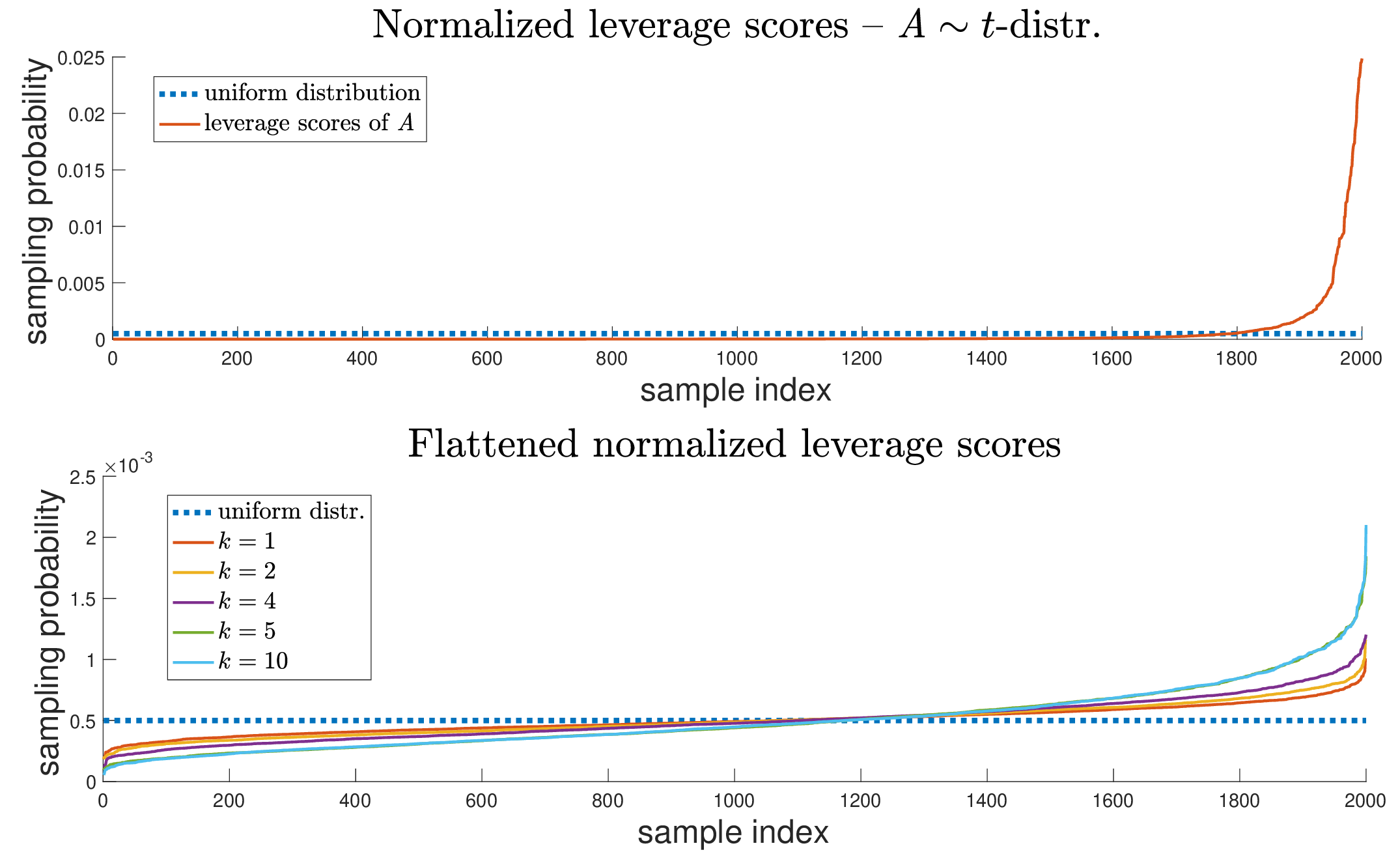}
    \caption{Flattening of leverage scores distribution, for $\Ab\sim t$-distribution.}
  \label{lvg_scores_tdistr_flat_fig}
\end{figure}

By Proposition \ref{norm_lvg_prop}, the quality of the flattening approximations of the local blocks depends on the sum of the local leverage scores $\Ll_i$.
To give better theoretical guarantees for our approach and analysis, we could make the assumption that $\Ll_i\approx d/k$ for each $i$. Such an assumption is weaker to analogous assumptions in distributed sketching algorithms which assume that $\Ab$ has low coherence, \textit{e.g.} \cite{WRXM18,NDV19,GMM20,NVR22,AN24}. This is akin to assuming that $\gamma_i\approx d/N$ for each $i\in\N_k$. We believe that it is not possible to improve on the local flattening algorithmically without exchanging information or aggregating the data a priori, as the objective is to flatten the scores of the collective $\Ab$. Further investigating this is worthwhile future work. We alleviate this concern by oversampling according to an appropriate misestimation factor $\betah$.

Before we present our main result, we need to show that $\Omb$ is close to being a uniform sampling (w.r.) matrix of $R$ out of $N$ rows. In Proposition \ref{prop_balls_bins}, we show this by applying Chebyshev's inequality to the balls into bins problem. Specifically, it shows that w.h.p. the sampling of $\Scal$ in $\I$ is close to a sampling of $\bigcup_{i=1}^k\Scal_i$, \textit{i.e.} the cardinality of sampled indices of $\Scal$ that lie in any $\I_i$ is not far from $\#\Scal_i=r$. The important factors to note are that all sampling trials in both scenarios are uniform, identical, independent and with replacement.

\begin{Prop}
\label{prop_balls_bins}
Partition the sampled index set $\Scal$ into ordered partitions $S_i$ of $\I$ according to $\I_{\{k\}}$, \textit{i.e.} $S_i=\Scal\bigcap\big(\bigcup_{l=1}^R\I_i\big)$. Then, for any $i\in\N_k$: $\Pr\big[|\#S_i-r|\geqslant10\big]\leqslant1/100$.
\end{Prop}

\begin{Rmk}
In order to get an exact global sampling index set for $\I$, the coordinator could determine $\Scal$ locally and then request the nodes to send their respective projected rows, or a subset of corresponding cardinality. In essence, this is similar to the ``unique sampling matrix $\Rb$'' of the block-SRHT \cite{BBGL23}.
\end{Rmk}

Next, we provide our main result regarding the $\ell_2$-s.e. of the aggregated dataset $\Ab$, through distributed local sketching.

\begin{Thm}
\label{subsp_emb_thm_global}
Let $\Pb_{\{k\}}$ of Algorithm \ref{alg_distr_sketch} be random unitary matrices, and $\betah=\frac{k}{d}\cdot\min_{i\in\N_k}\left\{\Ll_i\right\}$. Then, for $\delta>0$ and $R=\Theta\left(d\log{(2d/\delta)}/(\betah\epsilon^2)\right)$, the sketching matrix $\Sbwh$ of the global $\Ab$, satisfies \eqref{subsp_emb_id} with probability at least $1-\Theta(\delta)$.
\end{Thm}

\begin{Cor}
\label{subsp_emb_cor}
Consider $\frac{1}{\sqrt{r}}\Gb_{\{k\}}$ rescaled random Gaussian matrices, and perform Gram-Schmidt to each projection to obtain $\Pbh_{\{k\}}$. Then, for $\delta>0$ and $R=\Theta\left(d\log{(2d/\delta)}/(\betah\epsilon^2)\right)$, the sketching matrix $\Sbwh$ of the global $\Ab$, satisfies \eqref{subsp_emb_id} with probability at least $1-\Theta(\delta)$.
\end{Cor}

We point out that the failure probability of Theorem \ref{subsp_emb_thm_global} is higher than that of Theorem \ref{lvg_score_se_thm}, as there is also a source of error from Proposition \ref{prop_balls_bins} and the flattening of the leverage scores. Therefore, the higher $k$ is, the greater the failure probability is. Experimentally though, we observe that the increase in error is not drastic. Furthermore, if we were to assume that $\Ll_i\approx d/k$, \textit{i.e.} the local datasets are homogeneous w.r.t. $\Ub$, then the misestimation factor $\betah$ would be close to 1. For the general setting we are considering, it is best to avoid such assumptions in practice.

An interesting question is whether one can estimate $\Ll_{\{k\}}$, without directly aggregating or sharing the data. If so, the flattening of the scores through the local sketches could potentially be improved. Alternatively, each $\Scal_{\{k\}}$ can have proportional size to the corresponding $\Ll_i$, which would remove the oversampling factor of $1/\betah$ from Theorem \ref{subsp_emb_thm_global}.

Furthermore, the approach of Corollary \ref{subsp_emb_cor} suggests that each node performs a Gram-Schmidt process on its generated random matrix. The benefits of this is for the analysis of our technique, and can be avoided in practice, as these matrices satisfy $\E[\Sb^\top\Sb]=\Ib_N$.

It is worth noting the effect of parameter $k$. On one end of the spectrum when $k=1$, our method is equivalent to having a regular sketching matrix; \textit{e.g.} SRHT, as $\Pb_1=\Pbh\in\R^{N\times N}$ performs the random projection and $\Omb_1=\Ombh\in\{0,1\}^{R\times N}$ samples from the rotated/transformed data matrix. On the other end, $k$ should not be set too high w.r.t. $R$. Specifically, we want $r=R/k\gtrsim d$ as otherwise $\Sb_i\Ab_i$ locally would not satisfy Definition \ref{se_def}, and the oversampling factor $1/\hat{\beta}$ would be too high, \textit{i.e.} could result in $R>N$ in order to meet the sample complexity of Theorem \ref{subsp_emb_thm_global}. In a practical distributed setting, $k$ represents a physical constraint of the network, through which such an issue would not occur. Furthermore, in our experiments, we note that the performance of our method degrades gracefully for reasonable choices of $k$.

In the next proposition we further justify why we consider OSEs for our hybrid approach, other than the fact that they provide privacy benefits. The caveat though is that we require oversampling according to $\betah$, which seems to be inevitable with sampling based OSEs. We also note that even though $\Sbwh$ in Theorem \ref{hybrid_thm} cannot be a non-OSE $\ell_2$-s.e., $\Phib$ can be.

\begin{Prop}
\label{prop_non_obl}
We cannot apply non-oblivious local sketching without further communication rounds, or information sharing among the server nodes, in hope of getting a global sketch $\Abwh \coloneqq \Sbwh\Ab$ of $\Ab$ (according to \eqref{global_sketch}), w.r.t. \eqref{subsp_emb_id}.
\end{Prop}

Consider for instance the leverage score sketching matrix, which is not an OSE. In this case, the global properties that would need to be inferred is the orthonormal basis $\Ub$ of $\Ab$. For any fixed $\iota\in\N_k$, it is not possible to locally determine $\Ub_\iota$ \eqref{part_U} without any knowledge of $\Ab_{\{k\}}\backslash\{\Ab_\iota\}$, as the alignment of $\Ub_\iota$ with $\Ib_N$ depends on the entire data matrix $\Ab$. It is though possible to obtain such a subspace embedding by obtaining a (global) sampling scheme according to $\{\|\Ab_{(i)}\|_2^2\}_{i=1}^N$, though we would still require further rounds of communication. It is worthwhile investigating whether it is possible to devise a scheme which obtains an approximate global leverage score sampling sketch through local sampling, with few rounds of communication and restricted/limited information sharing.

% - - - - - - - - - - - - - - -
\section{Hybrid Sketching}
\label{hybr_sketching_sec}

In \cite{BP23}, the authors proposed sequentially applying two distinct sketches, where the first is analogous to Fig.~\ref{fig_schematic}, to form what they call a ``\textit{hybrid sketch}''. An illustration of such a hybrid sketching procedure is portrayed in Fig.~\ref{fig_schematic_hybrid}. There are potential scenarios where such sketches are useful, \textit{e.g.} there may be communication bottlenecks between the computational nodes and the central coordinator who wants to collect a description of the global dataset, though the coordinator wants to also compress it locally in order to save on computations which will then be done on the final sketch. Communication-efficiency is a key concern with all distributed algorithms, including federated ones \cite{AN24}. As was also noted in \cite{BP23}, it might be computationally feasible for server nodes to sample as many data points as possible; \textit{e.g.} $R/k$, and then reduce the dimension of the globally sampled data to the final sketch dimension $\Rt$ using another sketch with better error properties.

By communicating larger local sketches (through $\Sbwh$, with a higher $R$), the coordinator obtains a better approximation (Theorem \ref{subsp_emb_thm_global}), and can then apply a global sketch $\Phib\in\R^{\rho R\times R}$ to get a more succinct representation of $\Ab$, for $\rho\in(0,1)$. The initial reduced dimension is $\mu N$ for $\mu\in(0,1)$, \textit{i.e.} the global sketch dimension through $\Sbwh$ is $R=\ceil{\mu N}{}$. We denote the final target dimension of the global sketch by $\Rt=\ceil{\rho R}{}$, \textit{i.e.} $\Rt=\ceil{(\rho\mu) N}{}$.

\begin{figure}[h]
  \centering
    \includegraphics[scale=.18]{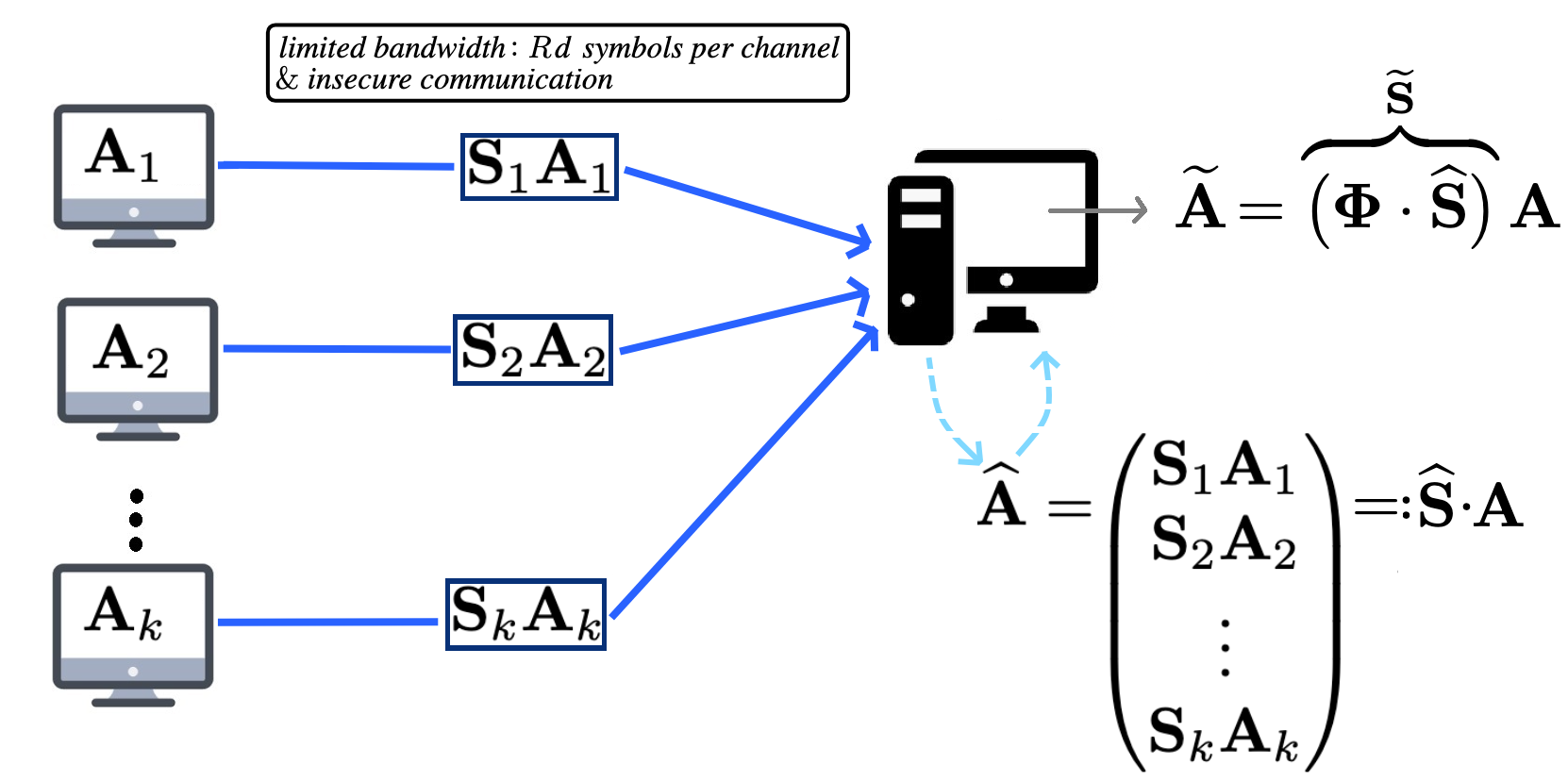}
    \caption{Schematic of hybrid sketching, with global sketching matrix $\Phib$.}
  \label{fig_schematic_hybrid}
\end{figure}

\begin{algorithm}[h]
\SetAlgoLined
  \KwOut{Sketch $\Abwt\in\R^{\Rt\times d}$, of the collective dataset $\Ab$}
  \underline{$k$ Distributed Servers}:
  \begin{enumerate}
    \item Locally sketch $\Ab_i\in\R^{n\times d}$ through $\Sb_i\in\R^{r\times n}$
    \item Communicate $\Abh_i\coloneqq\Sb_i\cdot\Ab_i\in\R^{r\times d}$ to the coordinator
  \end{enumerate}
  \underline{Coordinator}:
  \begin{enumerate}
    \item Aggregates $\Abwh = \Big[\Abh_1^\top \ \cdots \ \Abh_k^\top\Big]^\top$ as in Algorithm \ref{alg_distr_sketch}
    \item Performs sketching on $\Abwh$ through $\Phib\in\R^{\Rt\times R}$: $\Abwt\coloneqq\Phib\cdot\Abwh$
  \end{enumerate}
\caption{Distributed Hybrid Sketching}
\label{alg_hybrid_sketch}
\end{algorithm}

Below, we illustrate the embeddings which take place in hybrid sketching.
\[
\begin{tikzcd}
& {\image(\Ab)\ \subseteq\ \R^N} \arrow[rr, hook, "{\Sbwh}"] \arrow[rrrr, hook, bend left=40, dashed, "{\Sbwt = \Phib\circ\Sbwh}"'] & & \image(\Abwh) \subseteq\ \mathbb{R}^R \arrow[rr, hook, "{\Phib}"] & & {\mathbb{R}^{\Rt} \ \supseteq \ \image(\Abwt)}
\end{tikzcd}
\]

A benefit of applying the second dense sketching matrix $\Phib$ (unlike the sparse SJLT in \cite{BP23}), is that the resulting rows of the final sketch will contain linear combinations of elements from the global dataset $\Ab$, rather than just the local datasets $\Ab_{\{k\}}$. That is, we get a recombination of all the data vectors in the global dataset. This idea was motivated by the work of \cite{guha2000clustering}, which considers $k$-medians clustering. Specifically, the authors proposed an algorithm which achieves a constant factor approximation for the aforementioned problem, in a single pass and using small space. The algorithm essentially partitions the data points $X$ into $l$ groups $X_{\{l\}}$, and determines $k$ centers for each partition. Then, by only considering the $lk$ centers with reweighting according to the data points assigned to each center; say $X'$, it performs $k$-medians clustering on the smaller set $X'$ to obtain an approximate solution on $X$.

In our setting, by applying a global sketch $\Phib$ with better global properties than $\Sbwh$, the coordinator will thus have a better sketch in hand, as it is now projecting the global dataset $\Ab$ into a random lower-dimensional subspace formed by all the local datasets. This is also more beneficial, computationally and in terms of communication and privacy, to simply delivering $\Ab_{\{k\}}$ to the coordinator who will apply a global sketch to the aggregated $\Ab$. The composite sketching matrix after applying
\begin{equation}
\label{Phi_part}
  \Phib = \Big[\Phib^1 \ \cdots \ \Phib^k\Big]\in\R^{\Rt\times R} \quad \text{ where } \quad \Phib^i\in\R^{\Rt\times r}
\end{equation}
is then
\begin{equation}
\label{Phi_S_hybrid}
  \Sbwt \coloneqq \Phib\Sbwh = \Big[\overbrace{\Phib^1\Sb_1}^{\Sbwt^1} \ \cdots \ \overbrace{\Phib^k\Sb_k}^{\Sbwt^k}\Big]\in\R^{\Rt\times N}
\end{equation}
where $\Sbwt^i\in\R^{\Rt\times n}$, which $\Sbwt$ is dense. Consequently, each row of the resulting sketched matrix $\Abwh\coloneqq\Sbwt\Ab\in\R^{\Rt\times N}$ will be a linear combination of all $N$ elements of $\Ab$, unlike the distributed sketch $\Abwh=\Sbwh\Ab$ of Algorithm \ref{alg_distr_sketch}. Intuitively, this should therefore lead to a better \textit{global} approximation, while also satisfying the desired property that $\E[\Sbwt^\top\Sbwt]=\Ib_N$ in general, and $\E[\Sbwt]=\bold{0}_{\Rt\times N}$ when the local and global sketches are independent and each satisfy the respective property.

The work of \cite{BP23} only considered the case where uniform sampling is performed locally, and then a SJLT is applied by the coordinator. In terms of local computation uniformly sampling is the cheapest sketch one can consider, though it does not capture the spectral information of the local (nor global) datasets. This is why they then apply a superior sketch, the SJLT. In this section, we discuss how we can get better hybrid sketching results.

% - - - - - - - - - - - - - - -
\subsection{Embedding of Hybrid Sketching}
\label{hybr_embeddings}

Empirically, we observed that our approach is beneficial when $\Phib$ is a Rademacher sketch, and that in the case when $\Phib$ is an SRHT, we do not outperform mainstream sketching approaches. Furthermore, if the local $\Sb_{\{k\}}$ through $\Sbwh$ are SRHTs, the performance is very close to when $\Sb_{\{k\}}$ are Rademacher, and we have the benefit of accelerating the initial sketching procedure which is also done distributively. In the case when $N$ is relatively small, we also observed that selecting the local and the global sketching matrices to be a random Rademacher projection, leads to superior results. In conclusion, the most beneficial combination for large $N$ is when $\Sb_{\{k\}}$ are SRHTs, and $\Phib$ is a Rademacher random matrix, which we also justify theoretically. All in all, the ideal choice for sketching matrices in our setting of Fig.~\ref{fig_schematic_hybrid} and Algorithm \ref{alg_hybrid_sketch}, is to first apply SRHTs locally through $\Sb_{\{k\}}$, and then a Rademacher sketch on $\Abwh$ through $\Phib$.

As a first step, we prove Proposition \ref{hybrid_Prop}, which states that by sequentially applying two distinct sketching matrices, with modified error probability $\deltatil$ and distortion $\epst$ parameters, the resulting sketch satisfies \eqref{subsp_emb_id}. For our purpose, it seems that we cannot avoid considering the composition of two sketches separately in order to prove that $\Sbwt=\Phib\Sbwh$ is a valid $\ell_2$-s.e., as the application of $\Phib$ removes the assumption that the entries of $\Sbwt$ are i.i.d., which is usually needed for showing such results. Specifically, since there is a correlation between the entries of $\Sbwt$, this prohibits  us from giving a concentration guarantee on hybrid sketching based on prevalent matrix concentration bound techniques.

\begin{Prop}
\label{hybrid_Prop}
Fix the parameters $\epsilon,\delta\in(0,1)$ and $\Ub\in\R^{N\times d}$. Assume we are given sketching matrices $\Pib_1\in\R^{R\times N}$ and $\Pib_2\in\R^{\Rt\times R}$, both satisfying \eqref{subsp_emb_id} with parameters $\deltatil\coloneqq\delta/2$ and $\epst\coloneqq\sqrt{1+\epsilon}-1$ for $\image(\Ub)$ and $\image(\Pib_1\Ub)$ respectively. Then, the hybrid sketching matrix $\Sbh\coloneqq\Pib_2\Pib_1$ satisfies \eqref{subsp_emb_id} for $\image(\Ub)$ with parameters $(\epsilon,\delta)$.
\end{Prop}

\begin{Rmk}
Note that depending on the choices of $\Pib_1$ and $\Pib_2$, one can further optimize the parameters $(\epst,\deltatil)$ to get superior results on either the sample or time complexity of the resulting sketch, though there will always be a trade-off between two. Furthermore, one can also extend the above result to consider the composition of more than just two sketching matrices.
\end{Rmk}

\begin{Thm}
\label{hybrid_thm}
Consider the hybrid sketching approach according to Algorithm \ref{alg_hybrid_sketch}, where we first perform local sketching distributively through $\Sb_{\{k\}}$ according to \eqref{global_sketch}, and then apply a global sketch $\Phib$ on the aggregated sketches according to \eqref{Phi_S_hybrid}. As long as $\Phib$ and $\Sbwh$ both satisfy the $\ell_2$-s.e. property with parameters
$\deltatil=\delta/2$ and $\epst=\sqrt{1+\epsilon}-1$, then the resulting hybrid sketch satisfies the $\ell_2$-s.e. sketching matrix of the aggregated $\Ab$, with error probability $\delta$ and distortion $\epsilon$.
\end{Thm}

Experimentally (Sec.~\ref{exper_sec}), we observed that the best performance in terms of the empirical $\ell_2$-s.e. error $\epsilon$ according to \eqref{subsp_emb_id} occurs when the outer sketch is a random Rademacher matrix. Theoretically this choice is also sound, as it corroborates the facts conveyed in Table \ref{table_OSE}. Specifically, Rademacher  random projections require a smaller number of rows in comparison to the SRHT to guaranteed the same distortion. It is beneficial though in terms of complexity for the distributed servers to apply locally apply SRHTs. In summary, the outer sketching matrix $\Phib$ has a greater affect on the overall error $\epsilon$. By reducing the complexity required for the application of $\Phib$, we further accelerate the sketching of the global dataset $\Ab$, which was the main motivation for hybrid sketching \cite{BP23}. In practice, this also resulted in better empirical distortion error compared to Gaussian sketches.

There are also practical implications to applying two distinct binary random matrices, as by what we propose, the entries of the hybrid sketching matrix $\Sbwt$ follow a shifted binomial distribution\footnote{In the case where $\Sb_{\{k\}}$ and $\Phib$ are Rademacher matrices. For $\Sb_{\{k\}}$ SRHTs, this is also true empirically.} centered at 0. This is a consequence of the de Moivre–Laplace theorem, which can implies that $\Sbwt$ can be viewed as a discrete approximation
to a Gaussian random matrix, which are not easily realizable in practice. This gives further intuition and justification to our choices of $\Sb_{\{k\}}$ and $\Phib$.% The larger $N$ is, the more accurate the approximation.

% - - - - - - - - - - - - - - -
\subsection{Local SRHTs with a Global Rademacher}
\label{subsec_SRHT_Rrad}

Next, we focus on the ideal combination where $\Sb_{\{k\}}$ are SRHTs, and $\Phib$ is a Rademacher sketching matrix. For this analysis we make use of the sparse-SRHT \cite{TZCC22}, a special case of our distributed sketching approach through $\Sbwh$. The reason we do so, is to avoid the dependence on the oversampling parameter $1/\beta$ in the sample complexity $R$ of Theorem \ref{subsp_emb_thm_global}, which considers a variety of random projection matrices and not only the Randomized Hadamard Transform. It is also clear that unless we have a very low misestimation parameter $\beta$, then our $R$ is lower than that of \cite{TZCC22}. Lastly, we note that in many algorithms; including the SRHT, to remove the dependence on $\beta$ an assumption on its value is made, \textit{e.g.} $\beta=(2\ln(40Nd))^{-1}$ \cite{DMMS11,DM17}. Parameter $\beta$ is chosen such that an exact uniform sampling distribution results from the approximation of the flattened leverage scores, though it is significantly far from 1 for large $N$, which implies that the oversampling which takes place is relatively large.

Note that in the description of `Approach I' in \cite{TZCC22} for practical reasons an additional permutation matrix is applied to $\Ab$, and it is assumed that the same indexed samples are selected from each locally projected dataset, \textit{i.e.} $\Omb_i$ is the same for all $i\in\N_k$. As was aforementioned, the global permutation does not affect the theoretical result, though it could have practical benefits. To alleviate the caveat that we allow each server to perform a distinct $\Omb_i$, we can assume that a corresponding permutation is performed locally at each $\Ab_i$.

\begin{Thm}[\cite{TZCC22} Thm. 4]
\label{spSRHT_thm}
For any $\epsilon,\delta\in(0,1)$, $k$ the number of blocks, and $\Sbwh\in\R^{R\times N}$ the sparse-SRHT, the $\ell_2$-s.e. property of \eqref{subsp_emb_id} is satisfied with probability at least $1-\delta$ when
\begin{equation}
\label{red_dim_spSRHT}
  R \geqslant \frac{3k}{\epsilon^2}\left(\sqrt{d}+\sqrt{8\ln(3n/(k\delta))}\right)^2\ln(3d/\delta)\ .
\end{equation}

\end{Thm}

For our comparisons, we will consider the overall complexity of performing the resulting sketch $\Sbwt=\Phib\Sbwh$ on $\Ab$. The other measure we use to evaluate the sketching approaches is by the number of rows $R$ that are required to attain the same embedding guarantee in terms of \eqref{subsp_emb_def} and \eqref{subsp_emb_id}. To invoke Proposition \ref{hybrid_Prop}, for the remainder of this section we will assume that both sketching matrices $\Sbwh$ and $\Phib$ have design parameters $(\epst=\sqrt{1+\epsilon}-1,\deltatil=\delta/2)$, when compared to standard sketching techniques with parameters $(\epsilon,\delta)$.\\

\noindent \textbf{Comparison to the SRHT:} Recall that the main benefit of the SRHT, is that due to the structure of the random projection, it can be performed in $\ow(Nd\log N)$ oprations through recursive Fourier methods. Its main drawback, is the fact that it requires a larger embedding dimension $R$ to attain the same subspace embedding guarantee compared to other OSEs. Compared to most other randomized projections, $R$ in the SRHT has an additional logarithmic dependence on $N$, which is a consequence of applying the union bound in order to show that the leverage scores of the transformed basis are flattened.

The overall complexity of our application boils down to two matrix-matrix multiplications, by $\Sbwh$ and $\Phib$, which are of complexity $\ow(nd\log n)$ where $n=N/k$ for $k$ not too large, and $\ow(\Rt Rd)=\ow(N^2d\rho\mu^2)$ respectively. This complexity is then $\ow\big(Nd(\log(N/k)/k+N\rho\mu^2)\big)$ which is asymptotically less than that of the SRHT when
\begin{align*}
  N\rho\mu^2 < \log(N)-\frac{\log(N/k)}{k} = \log(N)\left(\frac{k-1}{k}\right)+\frac{\log(k)}{k} \approx \log(N)
\end{align*}
hence $\rho\mu^2<\frac{\log(N)}{N}$. By selecting the lowest possible $\mu$ in terms of \eqref{red_dim_spSRHT}, neglecting constant terms and setting $\delta=6/k$, we have an overall improvement when $\epsilon$ and $k$ satisfy
%we have an overall improvement when we select $\epsilon$ and $k$ such that
$$ \epsilon^6 > \frac{d\log(kd)\big(d+\log(k)\big)}{(N/k)^2}\ . $$
Even though this seems stringent and not satisfactory as this implies a relatively big $\epsilon$, it is possible to attain a small $\epsilon$ for large $N$. This is also where the main benefit of the SRHT arises, in the complexity of applying the projection. As we will see next though, our proposed hybrid sketch is superior in terms of the sample complexity required for an $\ell_2$-s.e., for reasonable parameters.

To get an overall improvement on the sample complexity, we need $R=kr$ for parameters $(\epst,\delta/2)$ to meet the lower bound of \eqref{red_dim_spSRHT}, \textit{i.e.}
\begin{align*}
  r \geqslant \frac{3}{\epst^2}\left(\sqrt{d}+\sqrt{8\ln(3n/(k\deltatil))}\right)^2\ln(3d/\deltatil) \geqslant \frac{6}{\epsilon^2}\left(\sqrt{d}+\sqrt{8\ln(6n/(k\delta))}\right)^2\ln(6d/\delta)
\end{align*}
and then require $\Rt$ to be small enough so that it is below the sample complexity required by the SRHT $Q= \Theta\big(d\log(Nd/\delta)\log(d/\delta)/\epsilon^2\big)$. Note that the sample complexity of the Rademacher sketch is independent of $N$, which is where our benefit comes from, compared to Hadamard based approaches. Specifically, as long as we enforce
\begin{equation*}
\label{Rt_bd}
 \Rt = \ow\big((d+\log(2/\delta))/\epst^2\big)\notag< \Theta\big(d\log(Nd/\delta)\log(d/\delta)/\epsilon^2\big)\ .
\end{equation*}
\begin{comment}
\begin{align}
\label{Rt_bd}
 \Rt &= \ow\big((d+\log(1/\deltatil))/\epst^2\big)\notag\\
  &< \ow\big((d+\log(2/\delta))/\epsilon^2\big)\notag\\
  &< \Theta\big(d\log(Nd/\delta)\log(d/\delta)/\epsilon^2\big)\ .
\end{align}
\end{comment}
To simplify our expressions, assume $\delta$ is fixed; \textit{e.g.} $\delta=0.1$, and note that $\epst/\epsilon<1/2$ for any $\epsilon>0$; when $\epst=\sqrt{1+\epsilon}-1$. Then, the requirement on $\Rt$ reduces to $\Omega(d)<\Theta\big(d\log(Nd)\log(d)\big)$, which holds true for large $N$. We note a minor caveat here though, which needs to be ensured in order to not have vacuous statements. For the SRHT based sketches, one needs to select parameters so that the reduced dimensions $R$ and $Q$ are respectively less than $N$. This is dependent on both $N$ and $d$, and we therefore do not further elaborate on this point.\\

\noindent \textbf{Comparison to the Gaussian Sketch:} The main benefits of the Gaussian sketch is its sample complexity, which matches that of the Rademacher sketch, and the fact that it is an OSE. From this end, it will always be slightly superior to our proposed approach, as our sketch $\Phib$ will require more samples, as $\epst<\epsilon$ and $\deltatil<\delta$. Asymptotically though, if we consider fixed $\epsilon$ and $\delta$, both approaches have the same sample complexity.

In terms of overall time complexity, not considering the fact that our computations are also carried out distributively, of our approach's number of operations is $\ow\big(Nd\big(\log(N/k)+N\mu^2\rho)\big)\big)$, while that of the Gaussian sketch is $\ow(N\Rt d)=\ow(N^2d\rho\mu)$. By dropping the $\ow$-notation, our approach is then superior when $\rho\mu(1-\mu)>\frac{\log(N/k)}{N}$ which holds true in general, unless one selects very small $\rho$ and $\mu$. In the case where $\rho$ and $\mu$ are very small though, it is likely that the required number of samples by the respective sketch (according to Table \ref{table_OSE}) may not be satisfied. Furthermore, even for relative small compression factors of, \textit{e.g.} $\rho=\mu=1/2$, our method is superior when $k>8$.\\

\noindent \textbf{Space requirement:} It is worth noting that our distributed approach is superior to any standard sketching approach in terms of space requirement. The largest matrix which is stored in one location at any given time throughout our procedure, is the initial global sketch $\Abwh\in\R^{R\times d}$ \eqref{global_sketch}, which is evidently smaller than $\Ab$. An artifact of this, is that our proposed approach is also beneficial in terms communication when we want to compute a sketch of distributed data, in the case when we would need to first aggregate the data in one server before applying the sketch.\\

\noindent \textbf{Network constraints:} Lastly, we outline the parameters selection when we have network imposed constraints, such as the number of servers, communication and storage budgets. We can assume that $k$ is fixed and all servers are homogeneous, with storage capacity of at least $n$ data points of length $d$. If we now assume a communication budget of $C_1$ symbols from each server to the coordinator, in order to maintain as much information as possible from the local sketches and utilize the allowable network communication, we set $\mu=C_1/(nd)$. If the coordinator has a storage budget of $C_2$ symbols, we have to set $\mu=\min\left\{C_1/(nd),C_2/(knd)\right\}$. Upon receiving the local sketches, the coordinator has a sketch $\Abwh\in\R^{R\times d}$ according to \eqref{global_sketch} where $R=\floor{\mu N}{}$. Then, if the final sketch $\Abwt$ has a further constraint of not exceeding $C_3$ symbols, we will require that the second sketching matrix $\Phib$ is of size $\Rt\times R$, where $\rho=C_3/(Rd)$ and $\Rt=\floor{\rho R}$. It is important to note that the choice of sketching method and associated parameters directly influences the accuracy guarantees $(\epsilon,\delta)$ of $\Abwh$. For this reason, we seek to retain as many data points as possible at each step.

% - - - - - - - - - - - - - - -
\section{Distributed Gradient Descent and Security}
\label{sec_distr_GD_sec}
% - - - - - - - - - - - - - - -
\subsection{Implications to Distributed Gradient Descent}

Our approach to locally sketching data matrices also has implications to distributed GD, which is arguably the most widely used technique in contemporary machine learning. Since its introduction \cite{RM51}, there has been an extensive body of work on stochastic GD (SGD), more recently in the distributed setting. Similar to our framework, random sampling and sketching algorithms have been utilized to show attainability of the same guarantees as SGD \cite{BWE20,CMPH23,CPH23b} by performing GD in parallel with local computations, where the servers compute gradients of subsets of the data.

An implication of {\cite[Thm. 3]{CPH23b}}, is that as long as the $\ell_2$-s.e. property is satisfied by the \textit{iterative sketching}\footnote{By \textit{iterative sketching}, we mean that at each iteration of the algorithm, we apply a new sketching matrix.} (or sampling) technique which takes place, we converge in expectation to the optimal solution \eqref{x_star_pr_lr} at a rate inversely proportional to the number of iterations. To establish that our sketching technique (Fig.~\ref{fig_schematic}) can be utilized to obtain the same convergence guarantee, it suffices to show that if every server performs a different sampling $\Omb_i^{[t]}$ at each iteration $t$, \textit{i.e.} a new global underlying sketching matrix $\Sbwh^{[t]}$ is applied at each iteration, the overall gradient is an unbiased estimate of the gradient $\nabla_\xb L_{ls}(\Ab,\bb;\xb^{[t]})$, which we show this in Proposition \ref{unb_grad}.

\begin{Prop}
\label{unb_grad}
By applying a different local sampling matrix $\big\{\Omb^{[t]}_i\big\}_{i=1}^k$ at each iteration $t$, we obtain an unbiased gradient estimate of the least squares objective function \eqref{x_star_pr_lr}.
\end{Prop}

\begin{Cor}[{\cite[Thm. 3]{CPH23b}}]
\label{conv_norm_cor}
Assume a bounded gradient on the least squares objective resulting from distributed GD with iterative sketching according to Algorithm \ref{alg_distr_sketch}, with step-size $\eta_t=1/(\alpha t)$ for a fixed $\alpha>0$. Then, after $T$ iterations:
\begin{equation*}
  \E\big[\|\xb^{[T]}-\xb^\star\|\big] \leqslant \ow(1/T).
\end{equation*}
\end{Cor}

It is worth mentioning that even though we directly invoke \cite[Thm. 3]{CPH23b} to give a convergence guarantee when the central server aggregates the gradients and performs distributed gradient descent, the formulation here is vastly different to that of previous works. In those works, the projections are performed on the aggregated data, which is the main contribution of this section, \textit{i.e.} we showed that you do not need to aggregate the data to get a global embedding. Furthermore, this application to distributed GD can be viewed as an extension of the prior works, where now the data is meant to be kept private from the coordinator rather than the servers. Additionally, here we require less overall storage, as the coordinator only needs to keep track of the updated estimates $\gbt^{[t]}$ and $\xbh^{[t+1]}=\xbh^{[t]}-\eta_t\cdot\xbh^{[t]}$, for $\eta_t$ an appropriate step-size.

Recall that the gradient of $L_{ls}(\Ab,\bb;\xb)\coloneqq\|\Ab\xb-\bb\|_2^2$ defined in \eqref{x_star_pr_lr} for a given $\xb^{[t]}$ at iteration $t$ of GD, is
\begin{equation*}
  g^{[t]} \coloneqq \nabla_\xb L_{ls}\big(\Ab,\bb;\xb^{[t]}\big) = 2\Ab^\top\big(\Ab\xb^{[t]}-\bb\big) \in \R^d
\end{equation*}
and since $L_{ls}$ is additively and linearly separable, for the \textit{partial gradients}
\begin{equation*}
  g_i^{[t]} \coloneqq \nabla_\xb L_{ls}\big(\Ab_i,\bb_i;\xb^{[t]}\big) = 2\Ab_i^\top\big(\Ab_i\xb^{[t]}-\bb_i\big) \in \R^d
\end{equation*}
defined for each data block pair $\left\{\left(\Ab_i,\bb_i\right)\right\}_{i=1}^k$,
it follows that
\begin{equation*}
  g^{[t]} = \sum_{j=1}^k g_i^{[t]}\ .
\end{equation*}

Considering our approach and the global sketching matrix $\Sbwh^{[t]}$ defined through the global projection matrix $\Pbh$ \eqref{global_sketch_S} at iteration $t$, we get
\begin{align*}
  \ga^{[t]} &\coloneqq \nabla_\xb L_{ls}\big(\Pbh\Ab,\Pbh\bb;\xb^{[t]}\big)\\
  &= 2(\Pbh\Ab)^\top\big(\Pbh(\Ab\xb^{[t]}-\bb)\big)\\
  &= 2(\Aba)^\top\big(\Aba\xb^{[t]}-\bba\big)\\
  &= \sum_{j=1}^k 2(\Aba_i)^\top\big(\Aba_i\xb^{[t]}-\bba_i\big)\\
  &= \sum_{j=1}^k \nabla_\xb L_{ls}\big(\Pb_i\Ab,\Pb_i\bb;\xb^{[t]}\big)\\
  &= \sum_{j=1}^k \ga_i^{[t]}\\
\end{align*}
where $\Aba\coloneqq\Pbh\Ab$ and $\Aba_i\coloneqq\Pb_i\Ab_i$ for each $i\in\N_k$. The vectors $\bba$ and $\bba_{\{k\}}$ are defined analogously. In the case where $\Pb_{\{k\}}$ are unitary, it follows that $\ga^{[t]}=g^{[t]}$ for each $t$, while when $\E\big[\Pb_i^\top\Pb_i\big]=\Ib_n$ for each $i$, we have $\E\big[\ga^{[t]}\big]=g^{[t]}$ for each $t$.

To carry out GD distributively, through iterative sketching, the local servers will perform a distinct $\big\{\Omb^{[t]}_i\big\}_{i=1}^k$ at each iteration $t$, to get different local sketches
\begin{equation*}
  \Abh_i^{[t]} \coloneqq \big(\sqrt{n/r}\cdot\Omb_i^{[t]}\Pb_i\big)\cdot\Ab_i = \Sb_i^{[t]}\cdot\Ab_i
\end{equation*}
of their local dataset. Note that at each iteration, the local label vectors $\bb_{\{k\}}$ are sketched accordingly, with the same corresponding sketching matrices. The pseudocode of this approach is summarized in Algorithm \ref{distr_SGD_alg}. To simplify the pseudocode presentation, in contrast to the above description, we apply the rescaling simultaneously with $\Pb_{\{k\}}$.

It is worth noting that Algorithm \ref{distr_SGD_alg} also has the following guarantee, in which the iterative sketching is through Algorithm \ref{alg_distr_sketch}; our distributed local sketching approach.

\begin{Cor}[{\cite[Cor. 3]{CPH23b}}]
\label{conv_regret_cor}
The expected regret of the least squares objective resulting from distributed GD with iterative sketching according to Algorithm \ref{distr_SGD_alg}, with a diminishing step-size $\eta_s$, converges to zero at a rate of $\ow(1/\sqrt{s}+R/s)$, \textit{i.e.}
\begin{equation*}
  \E\big[L_{ls}(\Ab,\bb;\xb^{[s]})-L_{ls}(\Ab,\bb;\xb^\star)\big] \leqslant \ow(1/\sqrt{s}+R/s) .
\end{equation*}
\end{Cor}

\begin{algorithm}[h]
\SetAlgoLined
  \KwIn{$r$, and a random initialization $\xbh^{[0]}\in\R^d$}
  \KwOut{Approximate solution $\xbh\in\R^d$ to \eqref{x_star_pr_lr}}
  \For{$i=1$ to $k$}%
    {  
      $1.0)$ The coordinator delivers $r$ and $\xbh^{[0]}$ to the $k$ servers\\
      \underline{$i^{th}$ node}:\\
      $\ 1.1)$ Generate a random $\Pb_i\in\R^{n\times n}$\\
      $\ 1.2)$ Locally transform the data:\\
      \begin{itemize}
        \item[--] $ \Aba_i\coloneqq\big(\sqrt{n/r}\cdot\Pb_i\big)\cdot\Ab_i$
        \item[--] $ \bba_i\coloneqq\big(\sqrt{n/r}\cdot\Pb_i\big)\cdot\bb_i$
      \end{itemize}
    }
  \For{$i=1$ to $k$}
    {
      \For{$t=0,1,2,\ldots$}
        {
          \underline{$i^{th}$ node}:\\
          $\ 2.1)$ Uniformly sample $r$ rows from $\Aba_i$, through $\Omb_i^{[t]}$:\\
            $\qquad - \ \ \Abh_i^{[t]}=\Omb_i^{[t]}\cdot\Aba_i$ and $\bbh_i^{[t]}=\Omb_i^{[t]}\cdot\bba_i$\\
          $\ 2.2)$ Determine the sketched partial gradient:\\
           $\qquad - \ \ \gh_i^{[t]} = 2\big(\Abh_i^{[t]}\big)^\top\big(\Abh_i^{[t]}\xb^{[t]}-\bbh_i^{[t]}\big)$\\
          $\ 2.3)$ Deliver $\gh_i^{[t]}$ to the coordinator\\
          \underline{Coordinator}:\\
          $\ 3.1)$ Determines $\eta_t$ and the \textit{global} sketched gradient:\\
           $\qquad - \ \ \gh^{[t]} = \sum_{j=1}^k\gh_j^{[t]}$\\
          $\ 3.2)$ GD Update: $\xbh^{[t+1]}\gets\xbh^{[t]}-\eta_t\cdot\gh^{[t]}$\\
          $\ 3.3)$ Delivers $\xbh^{[t]}$ to the $k$ servers
        }

    }
\caption{Distributed Iteratively Sketched GD}
\label{distr_SGD_alg}
\end{algorithm}

% - - - - - - - - - - - - - - -
\subsection{Security in Distributed Local Sketching}
\label{privacy_subsec}

Another major benefit of considering random Gaussian matrices is the security they provide. In this work, \textit{security} means that the nodes’ raw data $\Ab_i$ cannot be shared with nor recovered by the coordinator, or potential eavesdroppers. Under the assumption that $\Ab_i$ is randomly sampled from a distribution with finite variance, then the mutual information per symbol between $\Sb_i\Ab_i$ (sketch of $\Ab_i$ observed by the coordinator) and $\Ab_i$, has a logarithmic upper bound in terms of the variance; which approaches zero as $n$ increases or if $r$ is selected appropriately \cite{ZWL08,ZWL09,bloch2021overview,BP23}. This implies information-theoretic security and privacy of the local data blocks. Similar results were obtained in \cite{CMPH22,CMPH23} for $\Pb_{\{k\}}$ random unitary matrices, which made different assumptions on $\Ab_i$ and the distribution of $\Ub_i$.

\begin{Prop}
\label{security_proposition}
Assume that $\Ab_i$ is drawn from a distribution with finite variance. Then, the rate at which information about $\Ab_i$ is revealed by the compressed data $\Sb_i\Ab_i$ for $\Sb_i\in\R^{r\times n}$ a Gaussian sketch, satisfies $\sup\frac{I(\Ab_i;\Sb_i\Ab_i)}{nd}=\ow(r/n)\to0$. Specifically, the original $\Ab_i$ and the observed $\Sb_i\Ab_i$ are statistically independent, which means we obtain perfect secrecy for a small enough sketch dimension $r$.
\end{Prop}

This is of particular interest to FL, which allows training models across multiple decentralized devices or servers holding local data samples, without exchanging them. In standard FL, a common procedure is for the local servers to communicate the model after a few iterations and then the central server aggregates or averages the models which are delivered back, and this procedure is repeated. Here, the data is clearly not directly communicated, though communicating updated models may still be vulnerable to various privacy attacks, \textit{e.g.} membership inference attacks \cite{FJR15} and model inversion attacks \cite{SSSS17}. This is where deploying a procedure like ours both accelerates and enhances security of federated algorithms.

We note that in order to prove information-theoretic security guarantees for our approach, one needs to make some mild but necessary assumptions regarding the random projection utilized for constructing $\Sb_i$, and the data matrix $\Ab_i$ \cite{CMPH23}. For a fixed $i\in\N_k$, the message space $\M_i$ needs to be finite, which $\M_i$ in our case corresponds to the set of possible orthonormal bases of the column-space of $\Ab_i$. This is something we do not have control over, and depends on the application and distribution from which we assume the data is gathered. Therefore, we assume that $\M_i$ is finite. For this reason, we consider a finite multiplicative subgroup $(\Otil_{\Ab_i},\cdot)$ of $O_n(\R)$, which contains all potential orthonormal bases of $\Ab_i$.

Furthermore, this approach resembles the \textit{one-time pad}, one of the few encryption schemes known to provide perfect secrecy. The main difference between the two approaches, is that we work over the multiplicative group $(\Otil_{\Ab_i},\cdot)$ whose identity is $\Ib_n$, while the one-time pad is defined over the additive group $\big((\Z_2)^m,+\big)$ whose identity is $\vec{\mathbf{0}}_{m}$.

% - - - - - - - - - - - - - - -
\section{Experiments}
\label{exper_sec}

In this section we consider several experiments to corroborate what we presented on distributed local and hybrid sketching for $\ell_2$-embeddings, and provide numerical justification. All experiments were carried out on a single server, though the implementation of the local sketches through $\Sbwh$ \eqref{global_sketch_S} can be carried out distributively and/or in parallel by $k$ server nodes.

% - - - - - - - - - - - - - - -
\subsection{Experiments on Distributed Sketching}
\label{exper_sec_distr_sketching}

In the following experiments we considered the errors according to the $\ell_2$-s.e. error (Fig.~\ref{experiment_l2_error}) and relative regression error \eqref{epsilon_error_LS} (Fig.~\ref{regression_experiments}), for $\Ab\in\R^{18000\times 40}$ following a $t$-distribution. We considered $R$ varying from 30\% to 90\% of $N$, for different values of $k$, and $\Pb_i$ rescaled Gaussian matrices. Through these experiments, we convey that the difference in error is small, while we save on computing on the transformed blocks $\{\Pb_i\Ab_i\}_{i=1}^k$ by a factor of $(k_1/k_2)^2$ on the overall computation when we move from $k=k_1$ to a smaller $k=k_2$. By combining steps $1)$ and $2)$ of Algorithm \ref{alg_distr_sketch}, we require only $\ow(rndk)$ operations by each server. Additionally, more accurate approximations were observed with higher $N$.

\begin{figure}[h]
    \centering
    % First figure
    \begin{minipage}[t]{0.48\textwidth}
        \centering
        \captionsetup{type=figure}
        \includegraphics[scale=.16]{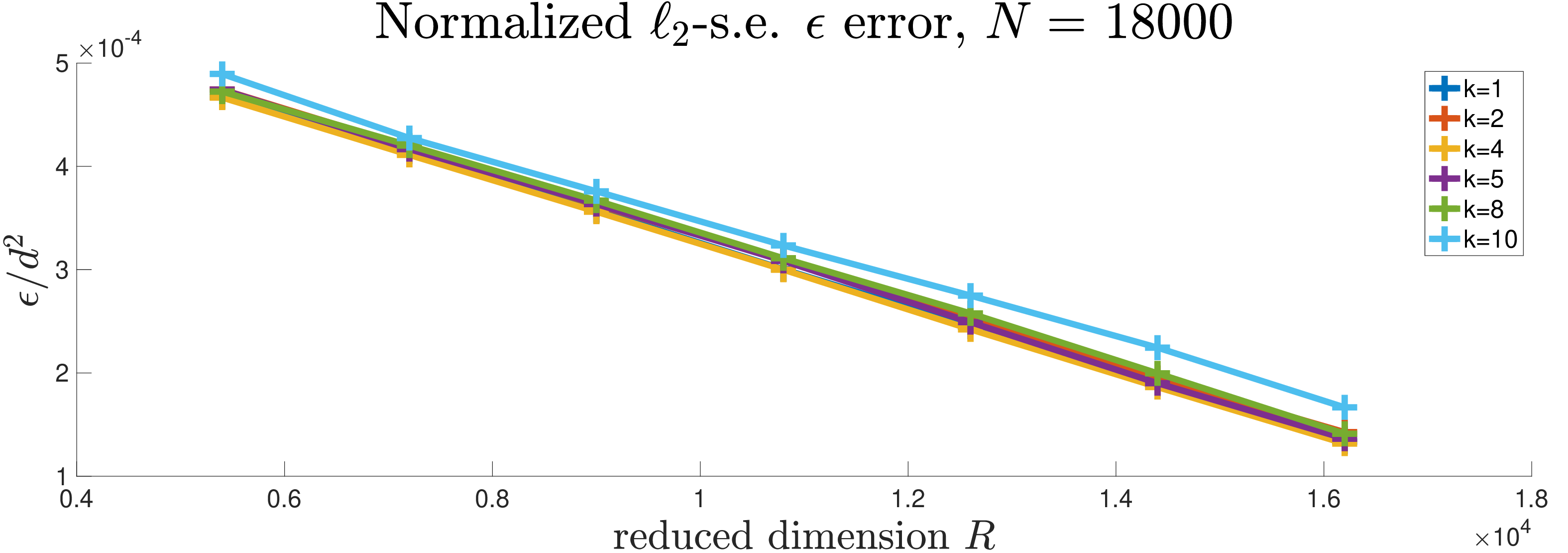}
        \caption{Normalized $\ell_2$-s.e. error \eqref{subsp_emb_id}, for varying $k$ and $R$.}
        \label{experiment_l2_error}
    \end{minipage}
    \hfill
    % Second figure
    \begin{minipage}[t]{0.48\textwidth}
        \centering
        \captionsetup{type=figure}
        \includegraphics[scale=.17]{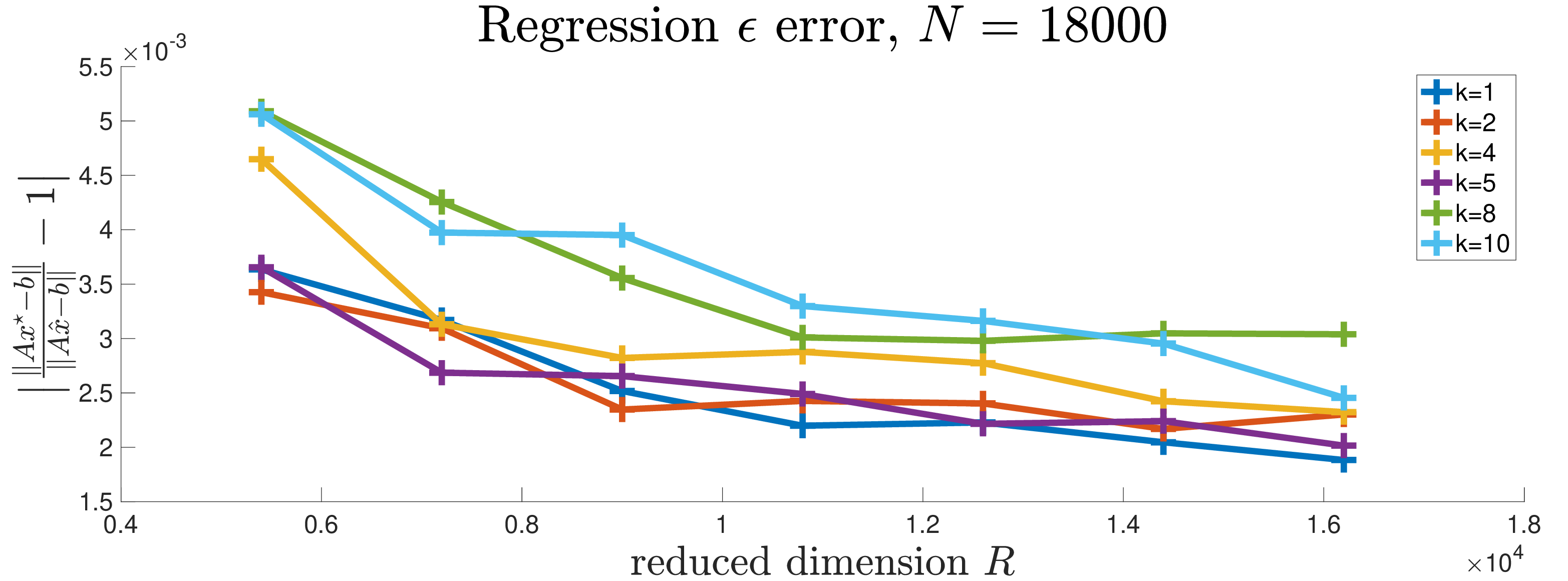}
        \caption{Regression error \eqref{epsilon_error_LS}, for varying $k$ and $R$.}
        \label{regression_experiments}
    \end{minipage}
\end{figure}

\begin{comment}
\begin{figure}[h]
    \centering
    \begin{subfigure}[b]{0.48\textwidth}
        \centering
        \includegraphics[scale=.16]{se_error_N18000_final.eps}
        \caption{Normalized $\ell_2$-s.e. error \eqref{subsp_emb_id}, for varying $k$ and $R$.}
        \label{experiment_l2_error}
    \end{subfigure}
    \hfill
    \begin{subfigure}[b]{0.48\textwidth}
        \centering
        \includegraphics[scale=.17]{regression_error_N18000_final.eps}
        \caption{Regression error \eqref{epsilon_error_LS}, for varying $k$ and $R$.}
        \label{regression_experiments}
    \end{subfigure}
    \caption{Side-by-side comparison of $\ell_2$-s.e. error and regression error.}
    \label{fig:side_by_side_errors}
\end{figure}
\end{comment}

\begin{comment}
\begin{figure}[h]
  \centering
    \includegraphics[scale=.18]{se_error_N18000_final.eps}
    \caption{Normalized $\ell_2$-s.e. error \eqref{subsp_emb_id}, for varying $k$ and $R$.}
  \label{experiment_l2_error}
\end{figure}

\begin{figure}[h]
  \centering
    \includegraphics[scale=.19]{regression_error_N18000_final.eps}
    \caption{Regression error \eqref{epsilon_error_LS}, for varying $k$ and $R$.}
  \label{regression_experiments}
\end{figure}
\end{comment}

% - - - - - - - - - - - - - - -
\subsection{FedAvg with Local Sketching on the Years Dataset}
\label{fedavg_expers}

In this experiment, we considered $N=10^4$ songs from the `Years Dataset' \cite{Lichman2011}, each comprised of $d=90$ audio features, to infer their release year. We wanted to evaluate the performance of our proposed method from Sec.~\ref{distr_local_sk_sec} when applied a priori to a FL task, such as FedAvg \cite{MMRHA17}. Specifically, each server applies a random matrix $\Pb_i$ to its dataset and then a uniform sampling matrix $\Omb_i$ before FedAvg is deployed; or updates the sampling matrix $\Omb_i^{[t]}$ at each iteration. The latter results in \textit{iterative sketching}, as at each iteration we have fresh sketches $\Sb_i^{[t]}\Ab_i$. The benefit of these approaches is that each server update is computed in a reduced time, as it is now considering a sketched dataset.

In Fig.~\ref{iter_sketching_fedavg} we considered $\mu=0.85$ and $k=8$, and used a diminishing step-size of $1/(L(t+1))$ for $L=\|\Ab\|_2^2$ the Lipschitz constant of the full gradient. We observe that with iterative local sketching, in this experiment, the estimated $\xbh^{[t]}$ converges at a faster rate compared to non-sketched FedAvg and an analogous scheme where sketching corresponds to naive uniform sampling. It is also evident that iteratively updating the local SRHTs (this can also be viewed as an extension of \cite{TZCC22}, as they only consider one global sketch) outperforms vanilla FedAvg, and justifies our choice for local SRHTs in our hybrid sketching approach. This is also the case for any of the other random projections considered.

\begin{figure}[h]
  \centering
    \includegraphics[scale=.24]{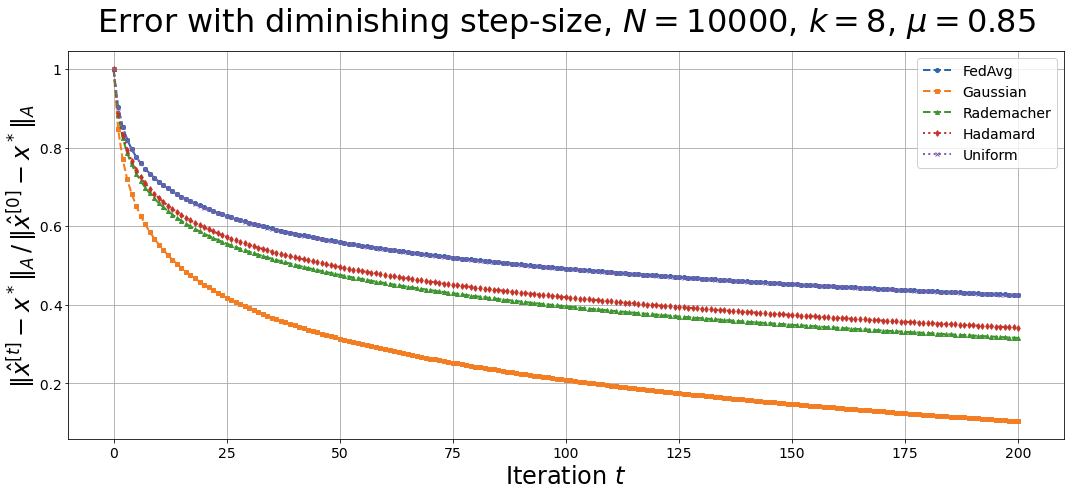}
    \caption{Error of $\xb^{[t]}$ at each iteration, for FedAvg with iterative sketching.}
  \label{iter_sketching_fedavg}
\end{figure}

% - - - - - - - - - - - - - - -
\subsection{Regression and $\ell_2$-s.e. Verification}

Similar to Sec.~\ref{exper_sec_distr_sketching}, for the following experiments we generated data matrices $\Ab\in\R^{N\times 40}$ following a $t$-distribution, where $N$ and $k$ are noted in the plots' titles. We indicate the overall reduction factor by $\tau\coloneqq \rho\mu$, \textit{i.e.} $\Rt=\tau N$. We considered an initial reduction factor of $\mu$ varying between $0.5$ and $0.85$ in increments of $0.05$, and the second reduction was by a factor of $\rho=0.7$. In all the experiments of this subsection, it is clear that our error degrades gracefully, as $\tau$ decreases.

In Fig.~\ref{experiment_l2_error_hybrid_8192}, we observe a clear benefit in terms of the empirically observed $\epsilon$ of \eqref{subsp_emb_id} when $\Phib$ is a Rademacher sketch, for $k=8$. There also seems to be a slight improvement when the local sketches are SRHTs, though when $\Phib$ is an SRHT, the hybrid sketch is inferior to standard sketching approaches. This justifies our choice for $\Phib$ and $\Sb_{\{k\}}$ from Sec.~\ref{subsec_SRHT_Rrad}.

\begin{figure}[h]
  \centering
    \includegraphics[scale=.18]{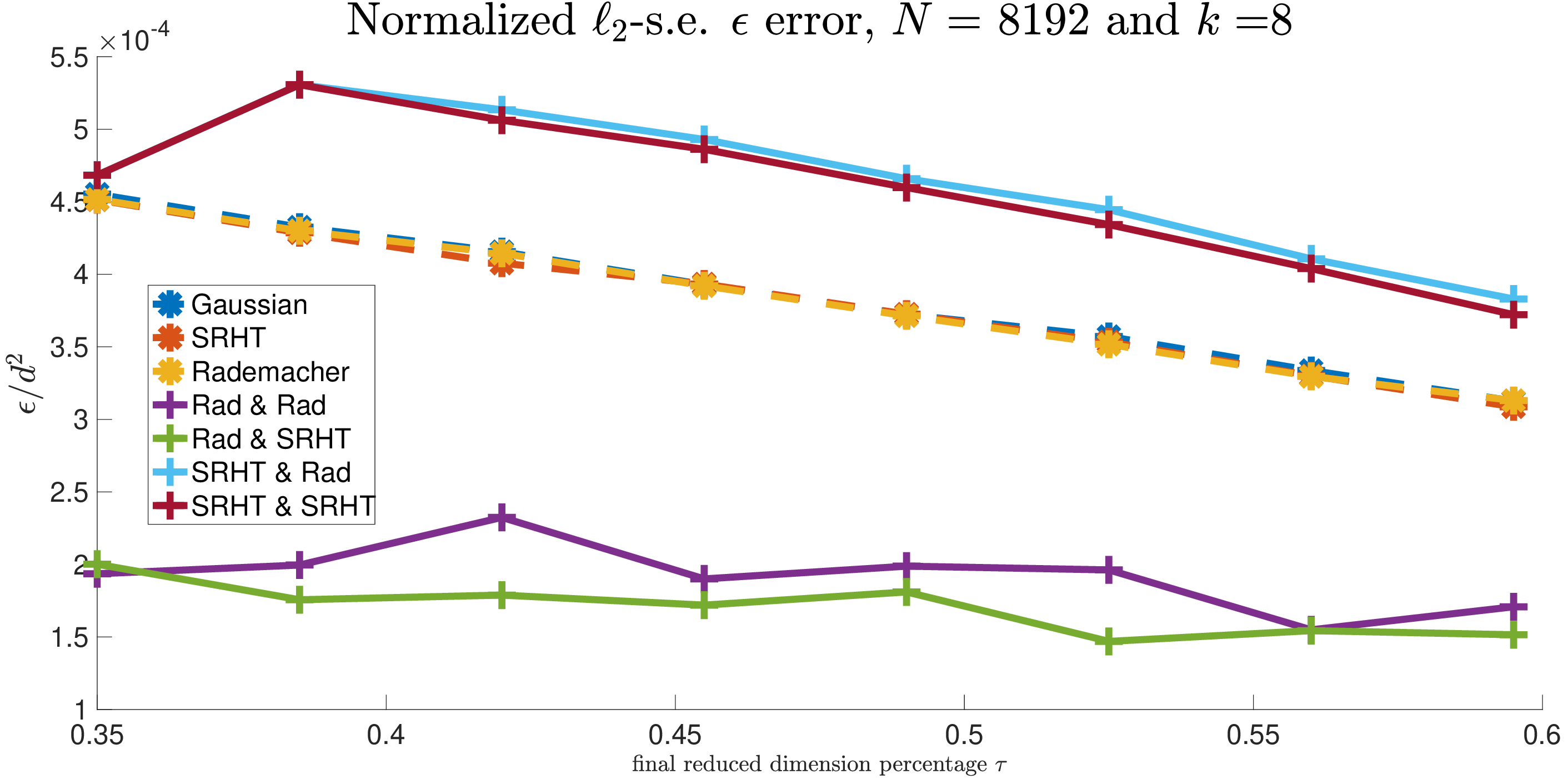}
    \caption{$\ell_2$-s.e. error \eqref{subsp_emb_id}, for fixed $\rho=0.7$ and $\mu$ ranging from $0.7$ to $0.85$.}
  \label{experiment_l2_error_hybrid_8192}
\end{figure}

For the same experiment, in Fig.~\ref{experiment_l2_error_hybrid_3000} we considered $k=32$ and $N=3000$. Here, due to the fact that $N$ and $n$ are both significantly smaller, we observe a noticeable distinction when $\Sb_{\{k\}}$ are Rademacher compared to SRHT. In this case though, since $n$ is quite small, we do not have a significant benefit in terms of time complexity when it comes to local sketching. Also, hybrid sketching outperforms standard sketching techniques. Similarly, in Fig.~\ref{experiment_l2_error_hybrid_8192_k32} we considered $N=8192$ with $k=32$, and analogous observations and conclusions were deduced. In this case, since $N$ was higher, we had better results. This is further justified by the fact that for higher $N$, the entries of $\Sbwt$ provide a better approximation to Gaussian random variables.

\begin{figure}[h]
    \centering
    % First figure
    \begin{minipage}[t]{0.48\textwidth}
        \centering
        \captionsetup{type=figure}
        \includegraphics[scale=.18]{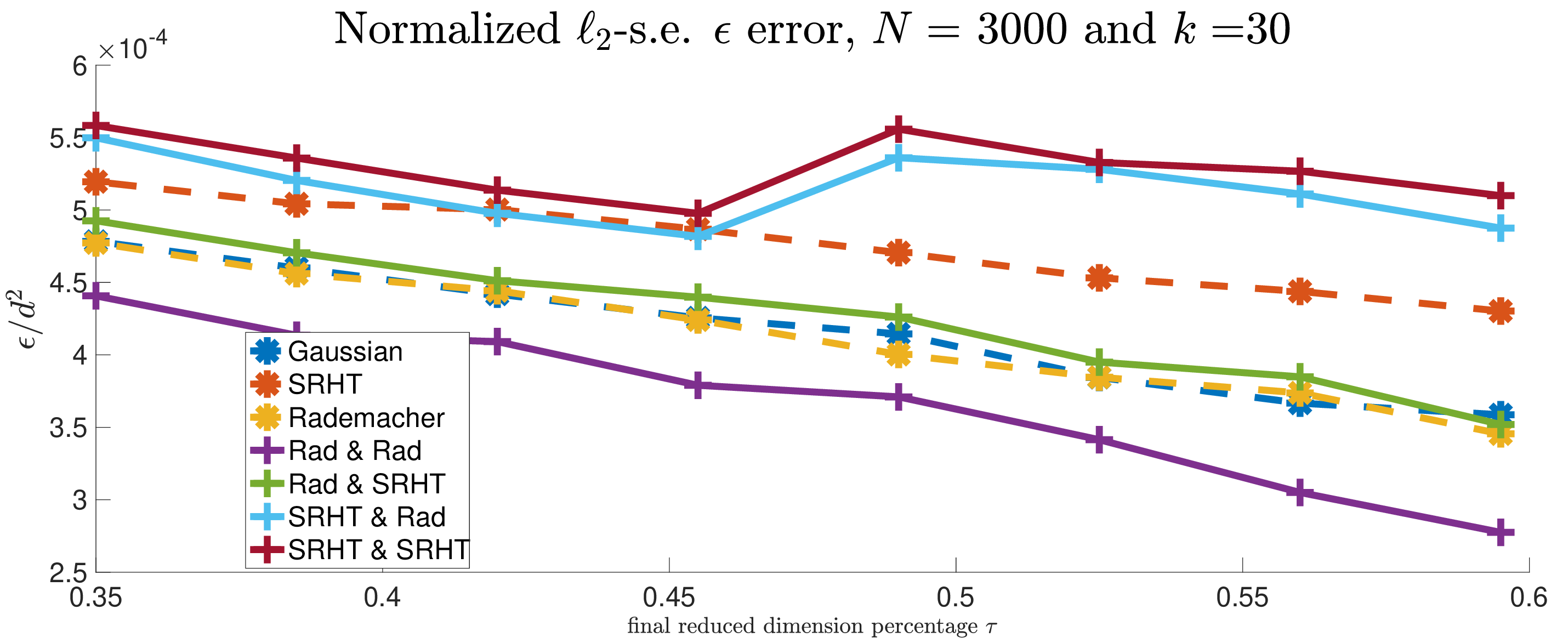}
        \caption{$\ell_2$-s.e. error \eqref{subsp_emb_id}, for $\rho=0.7$ and $\mu\in[0.7,0.85]$.}
        \label{experiment_l2_error_hybrid_3000}
    \end{minipage}
    \hfill
    % Second figure
    \begin{minipage}[t]{0.48\textwidth}
        \centering
        \captionsetup{type=figure}
        \includegraphics[scale=.18]{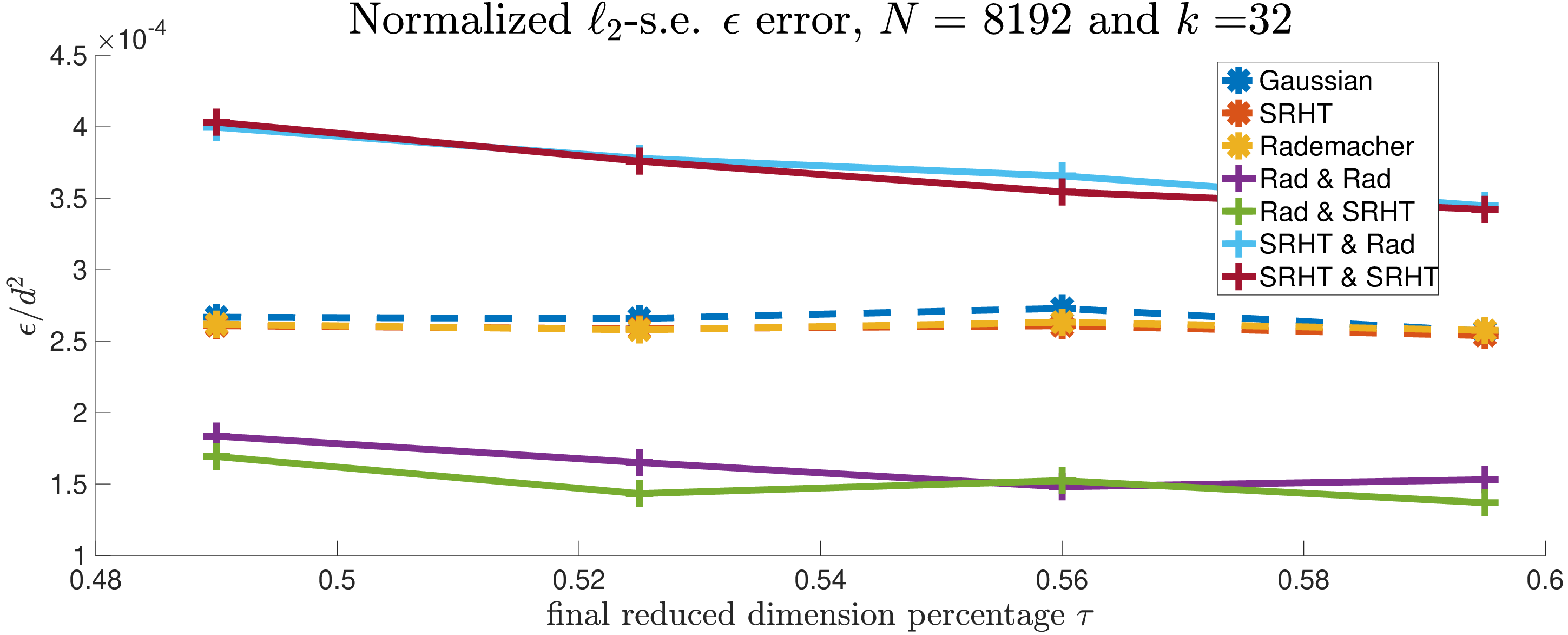}
        \caption{$\ell_2$-s.e. error \eqref{subsp_emb_id}, for $\rho=0.7$ and $\mu\in[0.7,0.85]$.}
        \label{experiment_l2_error_hybrid_8192_k32}
    \end{minipage}
\end{figure}

\begin{comment}
\begin{figure}[h]
    \centering
    \begin{subfigure}[b]{0.48\textwidth}
        \centering
        \includegraphics[scale=.18]{comparison_seven_3000_30_widerange.eps}
        \caption{$\ell_2$-s.e. error \eqref{subsp_emb_id}, $N=3000$.}
        \label{experiment_l2_error_hybrid_3000}
    \end{subfigure}
    \hfill
    \begin{subfigure}[b]{0.48\textwidth}
        \centering
        \includegraphics[scale=.18]{embedding_error_8192_40_2nd.eps}
        \caption{$\ell_2$-s.e. error \eqref{subsp_emb_id}, $N=8192$.}
        \label{experiment_l2_error_hybrid_8192_k32}
    \end{subfigure}
    \caption{Comparison of $\ell_2$-s.e. error for different $N$, with fixed $\rho=0.7$ and $\mu$ ranging from $0.7$ to $0.85$.}
    \label{fig:side_by_side_l2_error}
\end{figure}
\end{comment}

\begin{comment}
\begin{figure}[h]
  \centering
    \includegraphics[scale=.22]{comparison_seven_3000_30_widerange.eps}
    \caption{$\ell_2$-s.e. error \eqref{subsp_emb_id}, for fixed $\rho=0.7$ and $\mu$ ranging from $0.7$ to $0.85$.}
  \label{experiment_l2_error_hybrid_3000}
\end{figure}

\begin{figure}[h]
  \centering
    \includegraphics[scale=.22]{embedding_error_8192_40_2nd.eps}
    \caption{$\ell_2$-s.e. error \eqref{subsp_emb_id}, for fixed $\rho=0.7$ and $\mu$ ranging from $0.7$ to $0.85$.}
  \label{experiment_l2_error_hybrid_8192_k32}
\end{figure}
\end{comment}

% - - - - - - - - - - - - - - -
\subsection{Hybrid Sketching on the Years Dataset}

Considering the same data set as above \cite{Lichman2011}, we numerically solved \eqref{x_star_pr_lr} and compared our proposed approach to that of \cite{BP23} (Fig.~\ref{years_dataset_residual_error_fig}). We considered $k\in\{4,8,16\}$ and the same reduction factors $\rho$, $\mu$ as in Sec.~\ref{fedavg_expers}, and averaged each experiment over 3 different instances. It is evident from Fig.~\ref{years_dataset_residual_error_fig} that our choice for hybrid sketching had better overall performance than the hybrid sketching combination of \cite{BP23} (Uniform  Sampling and then SJLT) in terms of the residual error in the approximation of $\xbh$. Analogous results were also obtained for $N=2\cdot10^4$ and for each $k$ in the case of the ``Uniform+SJLT'' sketch.

\begin{figure}[h]
  \centering
    \includegraphics[scale=.28]{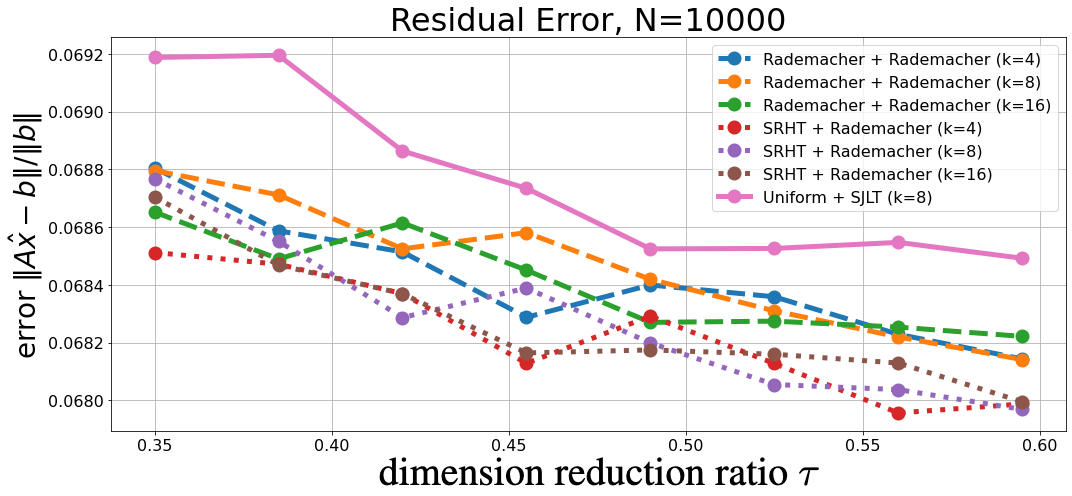}
    \caption{Normalized residual error for regression on the Years Dataset \cite{Lichman2011}.}
  \label{years_dataset_residual_error_fig}
\end{figure}

% - - - - - - - - - - - - - - -
\section{Future Directions and Concluding Remarks}
\label{concl_sec}

%{\purple \url{https://pages.cs.wisc.edu/~shuchi/courses/787-F09/scribe-notes/lec7.pdf}
%https://web.stanford.edu/class/cs265/Lectures/Lecture6/l6.pdf
%https://people.eecs.berkeley.edu/~jfc/cs174/lecs/lec6/lec6.pdf
%\url{https://people.scs.carleton.ca/~maheshwa/courses/3801/F23/Notes/balls.pdf}}

In the first portion of this work, we showed how we can obtain a global sketch of data distributed across a network in terms of a spectral guarantee, by performing local sketches and not aggregating until after the sketched matrices have taken place, satisfying desiderata $(\mathrm{a})$ and $(\mathrm{b})$ mentioned in the introduction. We also discussed the privacy aspect of this approach, as the local sketchings can provide security guarantees, meeting desideratum ($\mathrm{c}$). A potential future direction was mentioned in Sec.~\ref{analysis_subsec}, in terms of improving the flattening of the leverage scores distributively. Another interesting avenue is investigating the ideas presented in this paper, when used for sparser/distributed Johnson-Lindenstrauss transforms with partitions across the features of the data points.% Finally, there are potential methods which can tie this approach to that of \cite{BBGL23}, while still considering our decentralized setting, which can lead to further insights in distributed sketching.

Our distributive sketching approach was partly motivated by FL, where data is only available to local servers who may wish not to reveal their information. Our proposed approach can be utilized in conjunction with well-known first order federated algorithms to further accelerate them, which we also presented in our experiments. It also permits exchange of local sketches for better approximations by correlating the sketches, while also providing a global summary of the data if desired. Moreover, by obtaining a compressed summary of the global dataset, one can administer global updates of mainstream federated approaches through distributed or centralized first and second order methods, while still meeting the objective of keeping the local data secret. Specific applications include PCA and low-rank recovery \cite{GMCM20,SV23}, and subspace tracking \cite{NVR22}. To this extent, we also presented how our method can be utilized in distributed GD. Further investigation of these connections is worthwhile future work.

In the second part, we focused on hybrid sketching. Through hybrid sketching, we were able to satisfy all desiderata $(\mathrm{a})$-$(\mathrm{e})$. Though a simple idea, it has not been considered in previous works, other than few experiments carried out in \cite{BP23} with heuristic embeddings and reasoning, while lacking theoretical guarantees. As is evident, it is a powerful and beneficial technique which should garner more attention, specially in distributed environments. We extensively discussed the combination we believe is ideal for hybrid sketching and the aforementioned desiderata, and provided numerical justification. This idea can be also be extended to sequentially applying multiple sketching matrices, for instance in training a neural network comprised of multiple training layers, which poses a bottleneck in contemporary machine learning. Motivated by a $k$-medians clustering algorithm \cite{guha2000clustering}, it would be interesting to see where else these ideas could be leveraged.

% - - - - - - - - - - - - - - -

%\balance
\bibliographystyle{IEEEtran}
\bibliography{refs}
%\nobalance

% - - - - - - - - - - - - - - -
\appendices
% - - - - - - - - - - - - - - -

\section{Proofs of the Main Body}
\label{app_proofs}

In this Appendix, we include all the proofs missing from the main manuscript.
%In this Appendix, we include all the missing proofs of the main manuscript.

\begin{proof}{[Lemma \ref{exp_local_scores}]}
Recall that for any $j\in\I_i$:
\begin{equation*}
%\label{lvg_sc_expr}
  \ellh_j=\|\Ubh_{(j)}\|_2^2=\|\eb_i^\top\Ubh\|_2^2=\eb_j^\top\Ubh\Ubh^\top\eb_j\ .
\end{equation*}
It then follows that
\begin{align*}
  \E[\ellh_j] &= \E\left[\tr(\eb_j^\top\Ubh\Ubh^\top\eb_j)\right]\\
  &= \E\left[\tr(\eb_j\eb_j^\top\cdot\Ubh\Ubh^\top)\right]\\
  &\overset{\sharp}{=} \sum_{l\in\I_i}\frac{1}{|\I_i|}\cdot\tr\big(\eb_l\eb_l^\top\cdot\Ubh\Ubh^\top\big)\\
  &= \frac{1}{n}\cdot\tr\left(\sum_{l\in\I_i}\eb_l\eb_l^\top\cdot\Ubh\Ubh^\top\right)\\
  &\overset{\flat}{=} \frac{1}{n}\cdot\tr\left(\Ib_N\big|_{\I_i}\cdot\Ubh\Ubh^\top\right)\\
  &= \frac{1}{n}\cdot\tr\left(\Ubh_i\Ubh_i^\top\right)\\
  &= \frac{\Ll_i}{n}\ .
\end{align*}
In $\sharp$ we invoked the definition of expectation, to express $\E[\ellh_j]$ in terms of the leverage scores of the transformed block $\Ubh_i$. Further note that in $\flat$, the matrix $\Ib_N\big|_{\I_i}$ acts similar to a sampling matrix on the rows of $\Ubh\Ubh^\top$, as when we multiply with $\Ib_N\big|_{\I_i}$ we only retain the rows indexed by $\I_i$ of $\Ubh\Ubh^\top$; while the remaining rows are set to zero.
\end{proof}

\begin{proof}{[Proposition \ref{norm_lvg_prop}]}
We know that $\ellh_j\in[0,\Ll_i]$, and the normalized scores are $\ellb_j=\ellh_j/d$ for each $j\in\I_i$. By Lemma \ref{exp_local_scores}, it follows that
$$ \E[\ellb_j]=\E[\ellh_j/d]=\frac{1}{d}\cdot\E[\ellh_i]=\frac{\Ll_i}{nd}. $$
By applying Hoeffding's Inequality \cite{Mah16} for a fixed $\zeta\geqslant\zeta'$:
\begin{equation*}
  \Pr\left[|\ellb_j-\Ll_i/(nd)|<\zeta\right] > 1-2e^{-2\left(\zeta d/\Ll_i\right)^2} \geqslant 1-\xi \ .
\end{equation*}
\end{proof}

\begin{proof}{[Proposition \ref{prop_balls_bins}]}
We first note that $\bigcup_{l=1}^R\I_i$ is the multiset union of $R$ copies of the index set $\I_i$, as we are considering sampling with replacement. In our context, $\I_{\{k\}}$ represent the bins, and the allocation of the $R$ balls into the bins are represented by $\Scal$. The sub-multiset $S_i$ indicates which balls fell into the $i^{th}$ bin. Let $B_{ji}$ be the indicator random variable %r.v. of whether the $j^{th}$ ball lands into the $i^{th}$ bin, \textit{i.e.}
$$ B_{ji} = \begin{cases}
			1 & \text{if ball $j$ falls into bin $i$}\\
            0 & \text{otherwise}
		 \end{cases}
$$
for which $\Pr[B_{ji}]=1/k$ and $\E[B_{ji}]=1/k$, as $B_{ji}$ follows a Bernoulli distribution. Further define $Y^i\coloneqq\sum_{j=1}^RB_{ji}$ which follows a Binomial distribution, hence
\begin{equation*}
  \E\left[Y^i\right]=R/k=r \quad \text{and} \quad  \Var\left[Y^i\right] = (R-1)/R\ . % R\cdot(R-1)/R^2=(R-1)/R
\end{equation*}
It is clear that $Y^i=\#S_i$, so by Chebyshev's inequality the proof is complete.
\end{proof}

\begin{proof}{[Theorem \ref{subsp_emb_thm_global}]}
For each $i\in\N_k$ and every $j\in\I_i$, by Proposition \ref{norm_lvg_prop} we can assume that $\ellb_j\approx\Ll_i/(nd)$ (w.h.p.), where $\ellb_j$ is the sampling probability $\pit_j$ that row $j$ is sampled at each independent trial. Hence, Algorithm \ref{alg_distr_sketch} is now performing approximate leverage score sampling with distribution $\Uut_N=\bigcup_{\iota=1}^k\left(\bigcup_{j\in\I_\iota}\big\{\Ll_\iota/(nd)\big\}\right)$. The misestimation factor for sampling according to $\Uu_N$ instead of $\Uut_N$, is
\begin{equation*}
  \betah = \min_{\iota\in\N_k}\left\{\frac{\Ll_\iota/(nd)}{1/N}\right\} = \min_{\iota\in\N_k}\big\{\Ll_\iota\cdot(k/d)\big\}\ .
\end{equation*}

By Proposition \ref{prop_balls_bins}, the resulting $\Ombh$ through the local sampling matrices $\Omb_{\{k\}}$, is close to a uniform sampling matrix (w.r.) of $R$ out of $N$ elements. By applying the above conclusions to Theorem \ref{lvg_score_se_thm}, the proof is complete.
\end{proof}

\begin{proof}{[Corollary \ref{subsp_emb_cor}]}
Fix an $i\in\N_k$. By the Gram-Schmidt procedure on $\frac{1}{\sqrt{r}}\Gb_i$, the resulting $\Pbh_i$ is a random unitary matrix drawn from $O_n(\R)$. This holds true for all $i\in\N_k$. The claim then follows directly from Theorem \ref{subsp_emb_thm_global}.
\end{proof}

\begin{proof}{[Proposition \ref{prop_non_obl}]}
Assume that this is possible. Recall that for non-oblivious embeddings rely on the dataset $\Ab$, so in order to account for the heterogeneity in the global dataset $\Ab$, some statistics about the each $\Ab_i$ would need to be shared and compared in order to determine the global properties. This though contradicts the assumption that no communication takes place in order to determine the local sketching matrices $\Sb_{\{k\}}$.
\end{proof}

\begin{proof}{[Proposition \ref{hybrid_Prop}]}
For $\epsilon\in(0,1)$, the choice of $\epst$ implies $(1+\epst)^2=1+\epsilon$, and $(1-\epst)^2=(5+\epsilon-4\sqrt{1+\epsilon}) > 1-\epsilon$, \textit{i.e.} $\left[(1-\epst)^2,(1+\epst)^2\right] \subset \left[1-\epsilon,1+\epsilon\right]$. Therefore, we have $\epsilon=(1+\epst)^2-1>1-(1-\epst)^2$, which means that if an arbitrary matrix $S\in\R^{\Rt\times N}$ satisfies 
\begin{equation}
\label{subsp_emb_def_hybrid}
  (1-\epst)^2\cdot\|\xb\|\leqslant\|S\xb\|\leqslant(1+\epst)^2\cdot\|\xb\|
\end{equation}
for any $\xb\in\image(\Ub)$ with a high probability, then it satisfies \eqref{subsp_emb_def} with an even higher probability. For our case, we are considering $S=S_2\cdot S_1$ with $S_1:\R^N\to\R^R$ and $S_2:\R^R\to\R^\Rt$ satisfying \eqref{subsp_emb_def} for $\image(\Ub)$ and $\image(\Pib_1\Ub)$ respectively, each with distortion parameter $\epst$. By submultiplicativity of the $\ell_2$-norm, the product $S=S_2\cdot S_1$ satisfies \eqref{subsp_emb_def_hybrid}.

For simplicity and to maintain abbreviated notation, we consider the event of the hybrid sketch satisfying an $\ell_2$-s.e. with distortion error $\epsilon$, rather than the exact event of \eqref{subsp_emb_def_hybrid}. This suffices for our purpose, since
\begin{align*}
\label{implied_prob}
  \Pr\Big[(1-\epst)^2\cdot\|\xb\|\leqslant\|S\xb\|\leqslant&(1+\epst)^2\cdot\|\xb\|\Big] \geqslant \Pr\big[\|S\xb\|\leqslant_\epsilon\|\xb\|\big] \geqslant 1-\delta\ .
\end{align*}

For now, let $(\epsilon_1,\delta_1)=(\epsilon_2,\delta_2)=(\epst,\deltatil)$ denote the parameters of the corresponding sketching matrices. By the assumption on $\Pib_1$, for any $\xb$ we have $\Pr\left[\|\Pib_1\xb\|>_{\epsilon_1}\|\xb\|\right] < \delta_1$. Further define the set of all $\ell_2$-s.e. of $\Ub$ with distortion parameter $\epsilon$:
\begin{equation*}
  \Ecal_\Ub^\epsilon\coloneqq\left\{S_i\ : \ S_i \text{ is an } \epsilon \ \ell_2\text{-s.e. of } \Ub \right\}.
\end{equation*}

Let $\epsilon'=\max\big\{(1+\epsilon_1)(1+\epsilon_2)-1,1-(1-\epsilon_1)(1-\epsilon_2)\big\}=(1+\epsilon_1)(1+\epsilon_2)-1$, with a corresponding $\delta'$. The initial embedding of $\xb$ in the probability events is denoted by $\yb=S\xb$. By the law of total probability, it follows that
\begin{align*}
  \Pr\big[\|\Pib_2\Pib_1\xb\|>_{\epsilon'}\|\xb\|\big] &= \sum_{S\in\Ecal_\Ub^{\epsilon'}} \Pr\big[\Pib_1=S\big]\cdot\Pr\big[\|\Pib_2 S\xb\|>_{\epsilon'}\|\xb\| \ : \ \Pib_1=S \big]\\
  &\quad + \sum_{T\notin\Ecal_\Ub^{\epsilon'}} \Pr\big[\Pib_1=T\big]\cdot\Pr\big[\|\Pib_2 T\xb\|>_{\epsilon'}\|\xb\| \ : \ \Pib_1=T \big]\\
  &\overset{\natural}{<} \sum_{S\in\Ecal_\Ub^{\epsilon'}} \Pr\big[\Pib_1=S\big]\cdot\Pr\big[\|\Pib_2 \yb\|>_{\epsilon_2} \|\yb\| \ : \ \Pib_1=S \big] + \sum_{T\notin\Ecal_\Ub^{\epsilon'}} \Pr\big[\Pib_1=T\big]\cdot1 \\
  &\overset{\flat}{<} \sum_{S\in\Ecal_\Ub^{\epsilon'}} \Pr\big[\Pib_1=S\big]\cdot\delta_2 + \sum_{T\notin\Ecal_\Ub^{\epsilon'}} \Pr\big[\Pib_1=T\big]\\
  &\overset{\sharp}{=} \delta_2\cdot\Pr\big[\Ecal_\Ub^{\epsilon_1}\big] + \Pr\left[\big(\Ecal_\Ub^{\epsilon'}\big)^c\right]\\
  &\overset{\diamondsuit}{<} \delta_2\cdot 1 + \delta' \\
  &\overset{\heartsuit}{<} \delta_1+\delta_2\ .
\end{align*}
In $\natural$, we used the fact that the conditional probability in the second summand is always less than 1, and the fact that for an embedding with a smaller distortion we have a higher error probability, \textit{i.e.} $\delta'<\delta_2$ since $\epsilon'>\epst=\epsilon_2$. In $\flat$ we use the fact that the conditional probability in the first summand corresponds to the failure event of $\Pib_2$ not being an $\epsilon_2$ $\ell_2$-s.e., which occurs with probability less than $\delta_2$. In $\sharp$ we concisely represent the two summations, which are equal to the depicted events. In turn, we substitute their upper bounds which are respectively 1 and $\delta'$. Lastly, in $\heartsuit$ we use the fact that $\delta_1>\delta'$. By our choice of $\deltatil$, we know that $\delta_1+\delta_2=\delta$.

For our choice of $\deltatil$, the complementary probable event of the above is: $\Pr\big[\|\Pib_2\Pib_1\xb\|\leqslant_{\epsilon'}\|\xb\|\big] \geqslant 1-\delta$. Finally, for our choice of $\epst$ have $\epsilon'=\epsilon$, hence $\Pr\big[\|\Sbh\xb\|\leqslant_{\epsilon}\|\xb\|\big]\geqslant1-\delta$. This completes the proof.
\end{proof}

\begin{proof}{[Theorem \ref{hybrid_thm}]}
By Theorem \ref{subsp_emb_thm_global} and Corollary \ref{subsp_emb_cor}, we know that our distributed local sketching approach through Algorithm \ref{alg_distr_sketch} can be used to obtain an $(\epst,\deltatil)$ $\ell_2$-s.e. sketch $\Abwh=\Sbwh\Ab$ of $\Ab$. Then, by applying another $(\epst,\deltatil)$ $\ell_2$-s.e. sketching matrix $\Phib$ on $\Abwh$, according to Proposition \ref{hybrid_Prop}, we obtain an $(\epsilon,\delta)$ $\ell_2$-s.e. sketch $\Abwt=\Phib\cdot\Abwh=(\Phib\Sbwt)\cdot\Ab$ of $\Ab$. This completes the proof.
\end{proof}

\begin{proof}{[Proposition \ref{unb_grad}]}
Recall that the gradients of \eqref{x_star_pr_lr} and the modified sketched problem \eqref{x_til_pr_lr} at iteration $t$, are respectively
\begin{equation*}
  \gb^{[t]} = 2\Ab^\top(\Ab\xb^{[t]}-\bb) \quad \text{and} \quad \gbt^{[t]} = 2\Ab^\top(\Sb^\top\Sb)(\Ab\xb^{[t]}-\bb)\ .
\end{equation*}
By applying different local sampling matrices $\big\{\Omb^{[t]}_i\big\}_{i=1}^k$, the only thing that changes in \eqref{global_sketch_S} is the global sampling matrix $\Ombh$ at each $t$, which we denote by $\Ombh_{[t]}$, \textit{i.e.}
\begin{equation*}
  \Ombh_{[t]} \coloneqq \begin{pmatrix} \Omb_1^{[t]} & & \\ & \ddots & \\ & & \Omb_k^{[t]} \end{pmatrix}\in\{0,1\}^{R\times N}\ .
\end{equation*}
We denote the resulting sketching matrix by $\Sbwh_{[t]}$. Since in $\gbt^{[t]}$ the only source of randomness comes from $\Sbwh_{[t]}$, we have $\E\big[\gbt^{[t]}\big] = 2\Ab^\top\E\big[\Sbwh_{[t]}^\top\Sbwh_{[t]}\big](\Ab\xb^{[t]}-\bb)$. By modifying the proof of {\cite[Lemma 1]{CMPH23}}, we can show that $\E\big[\Sbwh_{[t]}^\top\Sbwh_{[t]}\big]=\Ib_N$. It suffices to show that $\E\big[\Ombh_{[t]}^\top\Ombh_{[t]}\big]=\Ib_N$. Let $\J_i$ denote the index set of $\Ombh_{[t]}$'s rows corresponding to $\Omb_i^{[t]}$, \textit{i.e.} $\J_i\subset\N_R$, $\bigsqcup_{\iota=1}^k\J_\iota=\N_R$, and $\#\J_i=r$ for each $i\in\N_k$. Then, $\Ombh_{[t]}\big|_{(\J_i)}$ is $\Omb_i^{[t]}$ presented in the submatrix corresponding to rows indexed by $\J_i$, and zeros everywhere else. We then have
\begin{align*}
  \E\left[\Ombh_{[t]}^\top\Ombh_{[t]}\right] &= \E\left[\sum_{i=1}^k\left(\Ombh_{[t]}\big|_{(\J_i)}\right)^\top\Ombh_{[t]}\big|_{(\J_i)}\right]\\
  &= \sum_{i=1}^k\E\left[\left(\Ombh_{[t]}\big|_{(\J_i)}\right)^\top\Ombh_{[t]}\big|_{(\J_i)}\right]\\
  &= \sum_{i=1}^k\Ib_N\big|_{(\J_i)} = \Ib_N
\end{align*}
and the proof is complete.
\end{proof}

\begin{proof}{[Proposition \ref{security_proposition}]}
This follows directly from Theorems 5.1 and 5.2 of \cite{ZWL08} when applied to our setting, and the definition of perfect secrecy \cite{bloch2021overview}.
\end{proof}

\section{Discrete Approximation to the Gaussian Sketch}
\label{discr_gaussian_appendix}

We include this appendix to give another practical benefit to our proposal on Hybrid sketching matrices, where $\Phib$ and $\Sbwh$ are binary random matrices, where the latter is also sparse and block diagonal (corresponding to the resulting global sketch of local distributed sketching). Additionally, what present here gives further intuition to our choice, and comparison to Gaussian sketching matrices. In essence, the hybrid sketching matrix $\Sbwt=\Phib\Sbwh$ can be interpreted as a discrete approximation to a Gaussian random matrix.

The main idea behind this is by an application of the de Moivre–Laplace theorem to the product of the sketching matrices $\Sb_{\{k\}}$ and $\Phib$. In what we present, the final sketching matrix approximates element-wise normalized Gaussian samples, though the entries themselves are not i.i.d. as $\Phib$ is applied globally; hence there is a correlation between the entries of the hybrid sketch, which prohibits us from giving a concentration guarantee on hybrid sketching.

The key justification to this is the de Moivre–Laplace theorem (Theorem \ref{deMoiv_Lapl}), a special case of the central limit theorem which states that the normal distribution may be used as an approximation to the binomial distribution under certain conditions. Below, we recall the de Moivre–Laplace theorem.% This result helps better motivate and provide intuition and justification to our choices of sketching matrices in our hybrid approach. %This is the main result we will use to justify that our choices of sketching matrices are appropriate for hybrid sketching.

\begin{Thm}[de Moivre–Laplace Theorem \cite{Ross97}]
\label{deMoiv_Lapl}
If $S_\nu$ denotes the number of successes that occur when $\nu$ independent trials, each
resulting in a success with probability $p$; are performed, then, for any $a<b$, we have 
\begin{equation}
\label{dML_eq}
  \Pr\left[a\leqslant\frac{S_\nu-\nu p}{\sqrt{\nu p(1-p)}}\leqslant b\right] \to \Phi(b)-\Phi(a)
\end{equation}
for $\nu\to \infty$, and $\Phi(\cdot)$ the cumulative distribution function of a standard normal random variable, \textit{i.e.}
\begin{equation*}
  \Phi(x) = \int_{\leqslant x} \frac{1}{\sqrt{2\pi}} e^{-y^2/2} \,dy\ .
\end{equation*}
\end{Thm}

In the case where all sketching matrices $\Phib$ and $\Sb_{\{k\}}$, are random (rescaled) Rademacher projections, each entry $\Sbwt_{ij}$ follows a shifted (rescaled) binomial distribution. The reason is that each product of any two Rademacher r.v.s (equivalently Bernoulli r.v.s with success probability $p=1/2$ -- $\Bern(1/2)$) results in a Rademacher r.v. Hence, the sum of the inner product between any pair of rows of $\Phib^{\lceil j/k\rceil}$ and $\Sb_{\lceil i/k\rceil}$ is itself a shifted binomial distribution. In the following corollary, we rephrase Theorem \ref{deMoiv_Lapl} to match our setting.

\begin{Cor}
\label{Cor_dML}  % corollary de Moivre–Laplace
Let $\yb,\zb\in\{-1,+1\}^\nu$ such that $\Pr\left[\yb_i=-1\right]=\Pr\left[\yb_i=+1\right]$ and $\Pr\left[\zb_i=-1\right]=\Pr\left[\zb_i=+1\right]$ for each $i\in\N_\nu$, i.e. $\yb$ and $\zb$ are vectors with random Rademacher entries. Then, for any $a<b$:
\begin{equation*}
\label{corollary_expr}
  \lim_{\nu\to\infty}\left\{\Pr\left[a\leqslant\frac{\langle\yb,\zb\rangle}{\sqrt{\nu}}\leqslant b\right]\right\} \to \Phi(b)-\Phi(a)\ .
\end{equation*}
\end{Cor}

\begin{proof}
We denote by $X_i$ a Rademacher r.v., and the sum of $\nu$ independent such variables by $X=\sum_{i=1}^\nu X_i$. Since $\langle\yb,\zb\rangle=\sum_{i=1}^\nu \yb_i\zb_i$ and $\Pr[\yb_i\zb_i=-1]=\Pr[\yb_i\zb_i=+1]=1/2$, we have $\langle\yb,\zb\rangle=X$.

Recall that we can transform a Rademacher r.v. $X_i$ to a $\Bern(1/2)$ r.v. though the transformation $X_i\mapsto\frac{X_i+1}{2}$. Denote the sum of $\nu$ the transformed random variables by $Y=\sum_{i=1}^\nu\frac{X_i+1}{2}=\frac{\nu}{2}+\frac{1}{2}X$. We can now apply what we have to the argument of \eqref{dML_eq}:
\begin{align*}
  \frac{S_\nu-\nu p}{\sqrt{\nu p(1-p)}} \mapsto \frac{Y-\nu/2}{\sqrt{\nu/4}} = \frac{\left(\frac{\nu}{2}+\frac{1}{2}X\right)-\frac{\nu}{2}}{\sqrt{\nu/4}} = \frac{\frac{1}{2}X}{\sqrt{\nu}/2} = \frac{X}{\sqrt{\nu}}\ .
\end{align*}
\end{proof}

From Corollary \ref{Cor_dML}, it is clear that for a larger $\nu$, which $\nu$ in our application corresponds to $R$, hybrid sketching involving random Rademacher matrices result in a finer approximation of a Gaussian distribution. One caveat when considering the SRHT as one of the projection matrices, is that not every entry is truly a random Rademacher entry, though empirically we observe that SRHT sketching matrices also perform well in hybrid sketching, and are preferable when considering time constraints in applying the larger sketching matrix $\Phib$.

In Figure \ref{fig_hybrid_sk_approx} we observe empirically that the probability density function of the hybrid sketches $\Phib\cdot\Sbwh$ with $\Sbwh$ being a random Rademacher matrix and $\Phib$ a random Rademacher matrix and then a SRHT, both result matrices whose entries are discrete approximations of a normal Gaussian distribution. It is worth noting at this point that the notion of a `discrete Gaussian distribution' has been used in the context of lattice based cryptography for learning with errors \cite{Reg09,Pei16,GMP20} and differential privacy \cite{CKS20}. In the context of lattice based cryptography, the discrete Gaussian distribution is defined as the Gaussian distribution restricted to a lattice coset, while in \cite{CKS20}; it resembles a Gaussian distribution which is supported only on the integers. While both scenarios closely resemble what we are presenting in Figure \ref{fig_hybrid_sk_approx}, they are both drastically different. The latter \cite{CKS20} is more similar to our discretization of a Gaussian normal distribution, though in our case the support is over the finite set $\left\{\sqrt{N/\Rt}\cdot\alpha\mid \alpha\in\{-R,-R+1,\ldots,R\}\right\}$.

Our approach has further benefits when compared to standard sketching with normal Gaussian matrices \cite{KVZ14}, as generating random Rademacher and SRHT sketching matrices and computing their product, is easier than generating actual Gaussian random matrices, which require more randomness. Furthermore, due to floating-point representation, in practice Gaussian sketches are always represented fine approximations.

\begin{figure}[h]
  \centering
    \includegraphics[scale=.2]{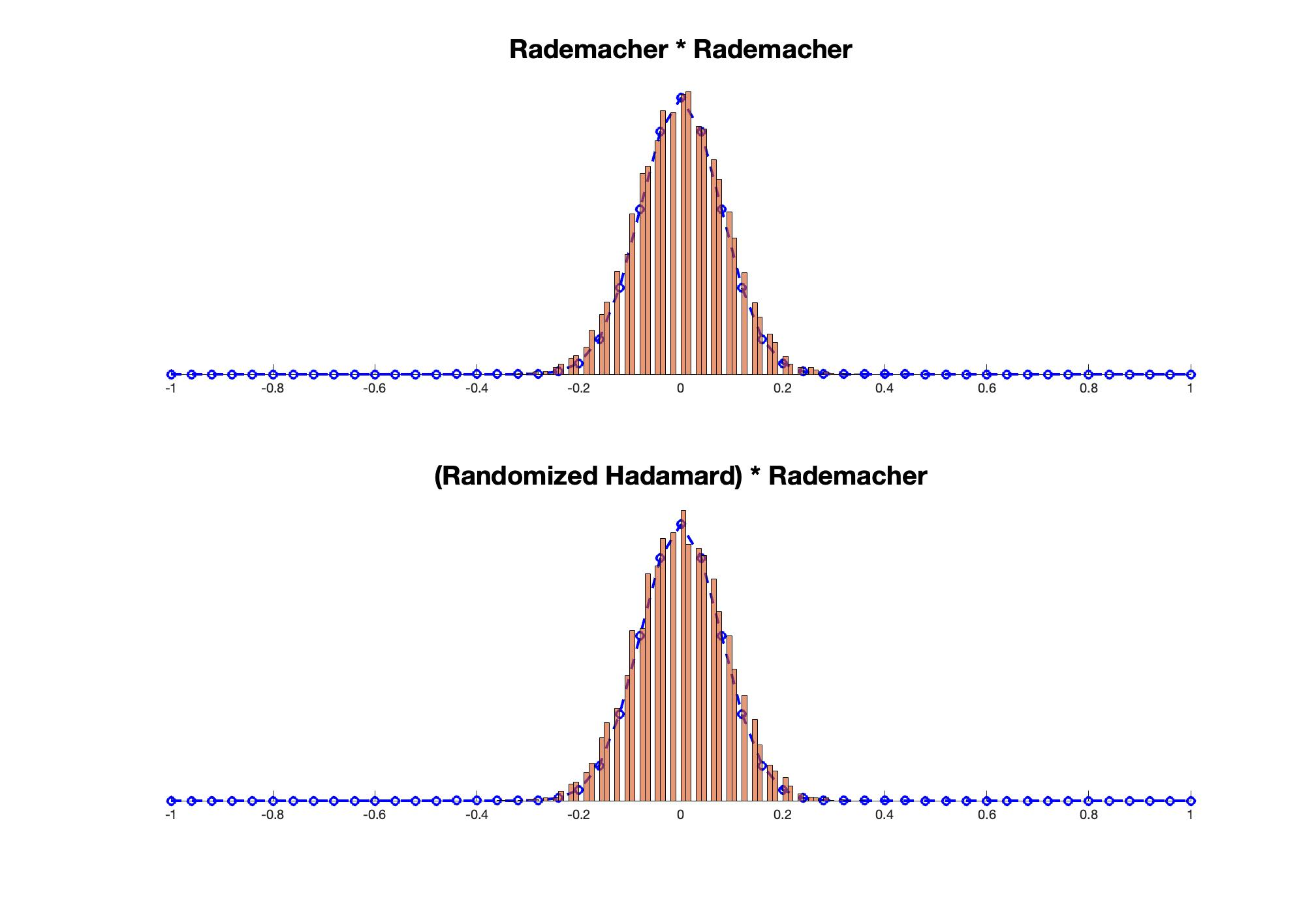}
    \caption{Empirical case where the product of two random Rademacher matrices, and of a randomized Hadamard and random Rademacher matrices, after appropriate rescaling, approximate a normal Gaussian distribution. In blue is the normalized Gaussian distribution, and in red we have a normalization of the histogram of the entries in the resulting two product matrices.}
  \label{fig_hybrid_sk_approx}
\end{figure}

\begin{Rmk}
\label{rmk_Gaussian_hybrid_sketch}
By Corollary \ref{Cor_dML} we conclude that the entries of the hybrid sketching matrix $\Sbwt$ \eqref{Phi_S_hybrid} can be interpreted as a discretization of a Gaussian normal distribution. Moreover, it is well known that such matrices are prime examples of Johnson-Lindenstrauss Transforms and $\ell_2$-s.e. embedding matrices.
%Moreover, by Theorem \ref{thm_gauss_embedding}, we know that a sketching matrix with normal Gaussian entries, satisfies the desired $\ell_2$-s.e. property \eqref{subsp_emb_id}. We can therefore deduce that in practice $\Sbwt$ implies a valid $\ell_2$-s.e. sketching matrix.
\end{Rmk}

\newpage
% - - - - - - - - - - - - - - -
\section{Flattening of leverage scores}
\label{app_flattening}

In this appendix, we provide further numerical experiments similar to Fig.~\ref{lvg_scores_tdistr_flat_fig}, which corroborate the fact that the local random projections $\Pb_{\{k\}}$ flatten the leverage scores of the global data matrix $\Ab$. The difference here, is that in Fig.~\ref{lvg_scores_sparse_flat_fig} we considered a sparse data matrix $\Ab$ which only had 10\% of its entries being nonzero, and in Fig.~\ref{lvg_scores_exp_flat_fig} the entries of $\Ab$ had varying magnitudes, in such a way that its leverage scores had an exponential-like trend. In both cases, $\Ab$ were generated randomly.

\begin{figure}[h]
  \centering
    \includegraphics[scale=.3]{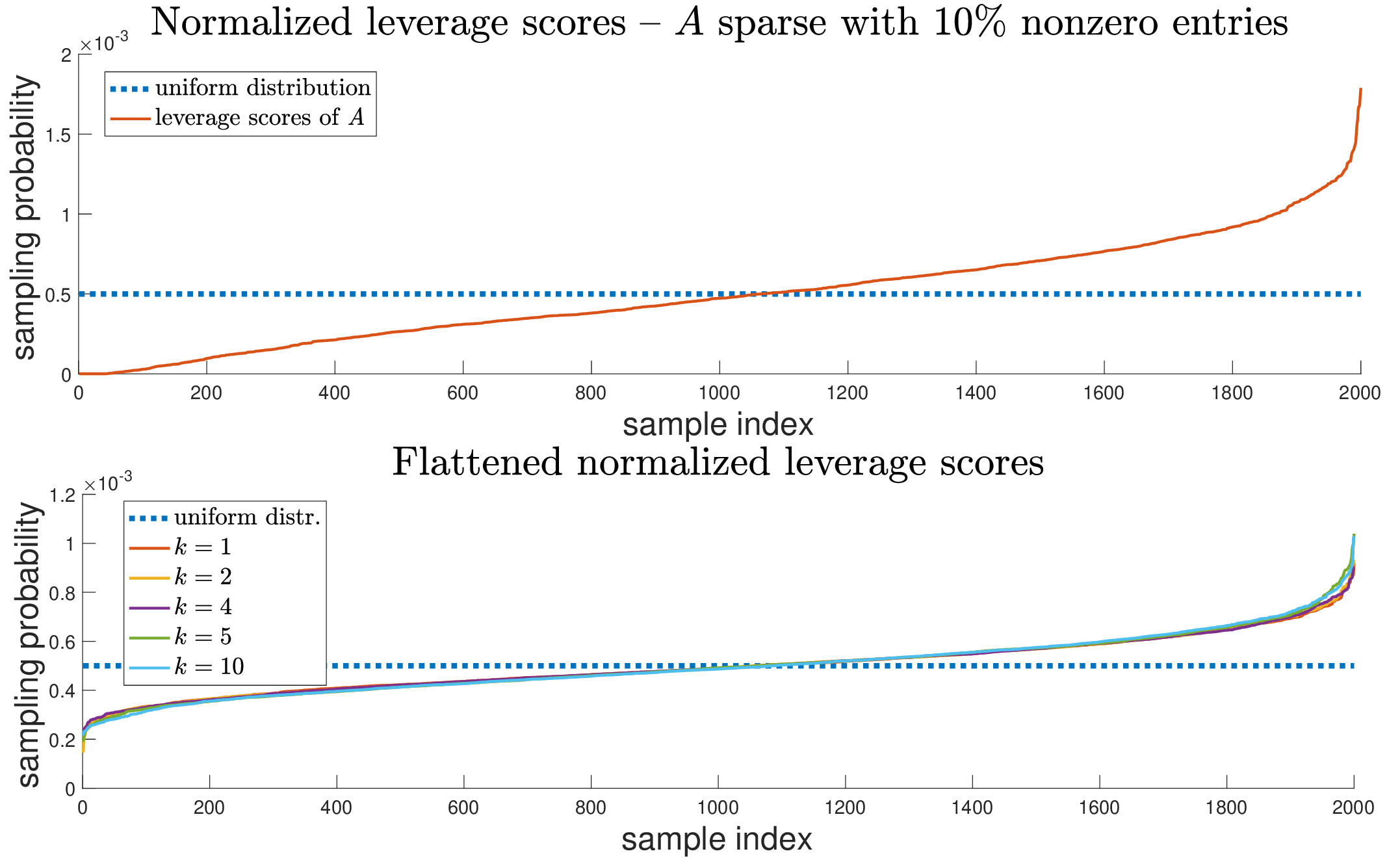}
    \caption{Flattening of leverage scores distribution, for $\Ab$ sparse with uniformly random entries, and $\nnz(\Ab)=Nd/10$.}
  \label{lvg_scores_sparse_flat_fig}
\end{figure}

\begin{figure}[h]
  \centering
    \includegraphics[scale=.3]{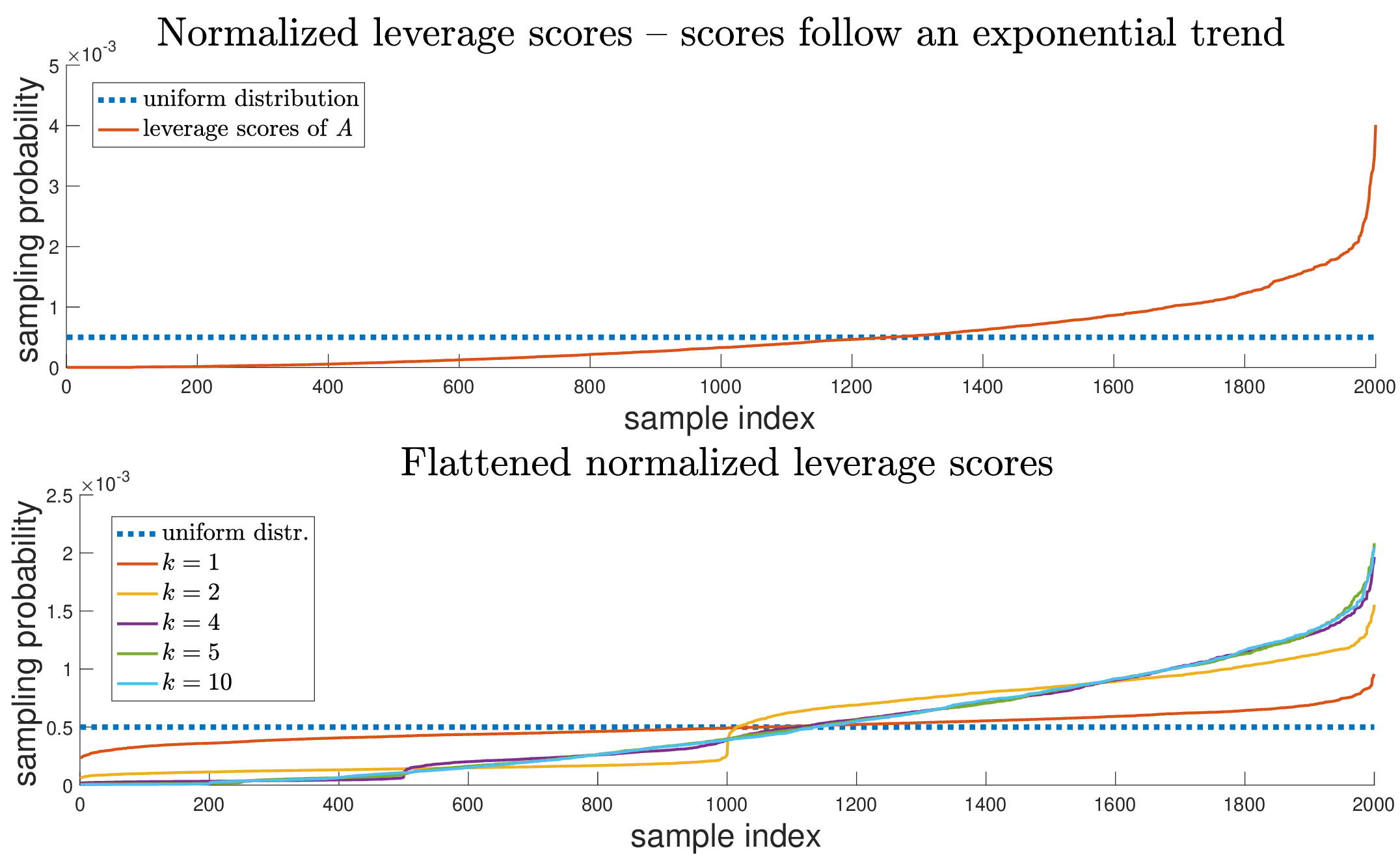}
    \caption{Flattening of leverage scores distribution, for random $\Ab$ with entries of varying magnitudes; by constant factors.}
  \label{lvg_scores_exp_flat_fig}
\end{figure}

\end{document}